\documentclass[english,oneside]{smfart}

\newcommand{\opn}[1]{\operatorname{#1}}

\usepackage{amsthm,mathtools,amssymb,textcomp,extarrows,bm,mleftright,graphicx,stmaryrd,scalerel,tikz,enumitem,float,booktabs,fontspec,mathspec}
\usepackage[dvipsnames]{xcolor}
\DeclareMathAlphabet{\mathbbb}{U}{bbold}{m}{n}

\DeclareMathOperator*{\sumbb}{\sum}
\DeclareMathOperator*{\sumtt}{\sum}
\DeclareMathOperator*{\prodtt}{\prod}

\usetikzlibrary{cd}
\mleftright

\makeatletter
\AtBeginDocument{%
 \let\glb@currsize\relax
}
\makeatother

\usepackage{mathrsfs}

\ExplSyntaxOn
\NewDocumentCommand{\definealphabet}{mmmm}
 {%
  \int_step_inline:nnn { `#3 } { `#4 }
   {
    \cs_new_protected:cpx { #1 \char_generate:nn { ##1 }{ 11 } }
     {
      \exp_not:N #2 { \char_generate:nn { ##1 } { 11 } }
     }
   }
 }
\ExplSyntaxOff
\definealphabet{bb}{\mathbb}{A}{Z}
\definealphabet{cal}{\mathcal}{A}{Z}
\definealphabet{frak}{\mathfrak}{A}{z}
\definealphabet{rm}{\mathrm}{A}{z}
\definealphabet{bf}{\mathbf}{A}{z}
\definealphabet{scr}{\mathscr}{A}{Z}
\definealphabet{tt}{\mathtt}{A}{Z}

\newcommand{\lto}{\longrightarrow}

\newcommand{\wtilde}{\widetilde}

\newfontface\logix {logix.otf}[Scale=1.0,NFSSFamily=logix]
\DeclareSymbolFont{Logix}{TU}{logix}{m}{n}
\newcommand{\lgl}[1]{\ensuremath{\Udelimiter 4 \symLogix "#1}}

\newcommand{\OpnDblParn}{\lgl{E31B}\relax}
\newcommand{\ClsDblParn}{\lgl{E35B}\relax}

\newcommand{\pparen}[1]{\left\OpnDblParn #1\right\ClsDblParn}

\newcommand{\uld}{\underline{d}}
\newcommand{\normd}{\lVert\uld\rVert}
\newcommand{\norm}[1]{\left\lVert#1\right\rVert}
\usepackage[pdfencoding=auto,psdextra]{hyperref}
\usepackage[nameinlink]{cleveref}
\Crefname{theorem}{Theorem}{Theorem}
\Crefname{conjecture}{Conjecture}{Conjectures}
\Crefname{lemma}{Lemma}{Lemmas}
\Crefname{definition}{Definition}{Definitions}
\Crefname{remark}{Remark}{Remarks}
\Crefname{question}{Question}{Questions}
\Crefname{proposition}{Proposition}{Propositions}
\Crefname{corollary}{Corollary}{Corollaries}
\Crefname{equation}{}{}
\Crefname{item}{}{}
\Crefname{example}{Example}{Examples}
\Crefname{proof}{Proof}{Proofs}
\newlist{thmenum}{enumerate}{1}
\setlist[thmenum]{label=\arabic*., ref=\thetheorem~(\arabic*)}
\crefalias{thmenumi}{theorem}
\newlist{propenum}{enumerate}{1}
\setlist[propenum]{label=\arabic*., ref=\theproposition~(\arabic*)}
\crefalias{propenumi}{proposition}
\newlist{lemenum}{enumerate}{1}
\setlist[lemenum]{label=\arabic*., ref=\thelemma~(\arabic*)}
\crefalias{lemenumi}{lemma}

\AddToHook{env/theorem/before}{\addvspace{3pt}}
\AddToHook{env/theorem/after}{\addvspace{3pt}}
\AddToHook{env/example/before}{\addvspace{3pt}}
\AddToHook{env/example/after}{\addvspace{3pt}}
\AddToHook{env/definition/before}{\addvspace{3pt}}
\AddToHook{env/definition/after}{\addvspace{3pt}}
\AddToHook{env/proposition/before}{\addvspace{3pt}}
\AddToHook{env/proposition/after}{\addvspace{3pt}}
\AddToHook{env/lemma/before}{\addvspace{3pt}}
\AddToHook{env/lemma/after}{\addvspace{3pt}}
\AddToHook{env/remark/before}{\addvspace{3pt}}
\AddToHook{env/remark/after}{\addvspace{3pt}}
\AddToHook{env/corollary/before}{\addvspace{3pt}}
\AddToHook{env/corollary/after}{\addvspace{3pt}}

\usepackage[frozencache]{minted}
\newmintinline[lean]{lean4}{fontsize=\footnotesize,bgcolor=white}
\usepackage{tcolorbox}
\tcbuselibrary{listings, minted, skins}

\tcbset{listing engine=minted}
\newtcblisting{leancode}[1][]{listing only, minted language=lean4, minted style=tango,
  colback=white!87!Emerald, enhanced, frame hidden, minted options={
  fontsize=\footnotesize, tabsize=2, breaklines, autogobble,mathescape},top=0pt,bottom=0pt,left=0pt,right=0pt,#1}

\usepackage[url=true,backend=biber,style=ext-alphabetic,hyperref=true,giveninits=true]{biblatex}
\makeatletter
\renewcommand*\subjclass[2][2020]{%
  \def\@subjclass{#2}%
  \@ifundefined{subjclassname@#1}{%
    \ClassWarning{\@classname}{Unknown edition (#1) of Mathematics
      Subject Classification; using '2020'.}%
  }{%
    \@xp\let\@xp\subjclassname\csname subjclassname@#1\endcsname
  }%
}
\@namedef{subjclassname@1991}{%
  \textup{1991} Mathematics Subject Classification}
\@namedef{subjclassname@2000}{%
  \textup{2000} Mathematics Subject Classification}
\@namedef{subjclassname@2010}{%
  \textup{2010} Mathematics Subject Classification}
\@namedef{subjclassname@2020}{%
  \textup{2020} Mathematics Subject Classification}
\@xp\let\@xp\subjclassname\csname subjclassname@2020\endcsname
\makeatother

\newtheorem{theorem}{Theorem}[section]
\newtheorem{example}[theorem]{Example}
\AddToHook{env/example/begin}{\crefalias{theorem}{example}}
\newtheorem{lemma}[theorem]{Lemma}
\AddToHook{env/lemma/begin}{\crefalias{theorem}{lemma}}
\newtheorem{remark}[theorem]{Remark}
\AddToHook{env/remark/begin}{\crefalias{theorem}{remark}}
\newtheorem{proposition}[theorem]{Proposition}
\AddToHook{env/proposition/begin}{\crefalias{theorem}{proposition}}
\newtheorem{definition}[theorem]{Definition}
\AddToHook{env/definition/begin}{\crefalias{theorem}{definition}}
\newtheorem{conjecture}[theorem]{Conjecture}
\AddToHook{env/conjecture/begin}{\crefalias{theorem}{conjecture}}
\newtheorem{corollary}[theorem]{Corollary}
\AddToHook{env/corollary/begin}{\crefalias{theorem}{corollary}}
\newtheorem{question}[theorem]{Question}
\AddToHook{env/question/begin}{\crefalias{theorem}{question}}

\AddToHook{env/assumption/begin}{\crefalias{theorem}{assumption}}

\Crefformat{enumi}{#2\textup{(#1)}#3}
\usepackage[verbose=silent]{microtype}
\usepackage{geometry}

\usepackage{orcidlink}

\title{$p$-adic Hahn series with sparse support}

\usepackage{orcidlink}
\author{Shanwen Wang\orcidlink{0000-0003-0228-1208}}
\address{School of Mathematics, Renmin University of China, No. 59 Zhongguancun Street, Haidian District, Beijing, 100872, China}
\address{Bigdata and Responsible Artificial Intelligence for National Governance, Renmin University of China, No. 59 Zhongguancun Street, Haidian District, Beijing, 100872, China}
\email{s\_wang@ruc.edu.cn}
\urladdr{https://shanwenwang.github.io/}
\author{Yijun Yuan\orcidlink{0000-0001-6571-6980}}
\address{Institute for Theoretical Sciences, Westlake University, No. 600 Dunyu Road, Sandun town, Xihu district, Hangzhou, Zhejiang Province, 310030, China}
\email{941201yuan@gmail.com}
\urladdr{https://yijunyuan.github.io/}
\begin{abstract}
Let $p$ be a prime number. We introduce a sparseness condition on the supports of $p$-adic Hahn series, and prove that this condition implies transcendence over $\breve{\mathbf Q}_p$, the completed maximal unramified extension of $\bfQ_p$. As an application, we prove the order-type conjecture of $\bfQ_p$-algebraic $p$-adic Hahn series with bounded support under the condition that the support has only finitely many accumulation points.

All results in this paper have been fully formalized in the Lean theorem prover (v 4.31.0), building over Mathlib.
\end{abstract}

\usepackage{letltxmacro}
\LetLtxMacro\oldttfamily\ttfamily
\DeclareRobustCommand{\ttfamily}{\oldttfamily\csname ttsize\endcsname}

\begin{document}\def\ttsize{\footnotesize}
\subjclass{11J61, 11J81, 11D88, 68V20}
\maketitle
\tableofcontents
\section{Introduction}
Let $p$ be a prime number. Let $\bfF_p$ be the finite field of $p$ elements, $\overline{\bfF}_p$ be an algebraic closure of $\bfF_p$. Let $\bfQ_p$ be the field of $p$-adic numbers, $\breve{\bfQ}_p=W(\overline{\bfF}_p)[p^{-1}]$ be the completed maximal unramified extension of $\bfQ_p$, $\breve{\bfZ}_p=W(\overline{\bfF}_p)$ be the ring of integers of $\breve{\bfQ}_p$, $\overline{\bfQ}_p$ be an algebraic closure of $\bfQ_p$, and $\bfC_p=\widehat{\overline{\bfQ}}_p$ be the field of $p$-adic complex numbers. We normalize the $p$-adic valuation on $\bfC_p$ and its subfields by $v_p(p)=1$. For any set $X$, we denote by $\opn{card}(X)$ the cardinality of $X$.

\subsection{$p$-adic transcendence via $p$-adic Hahn series}
A $p$-adic Hahn series is a generalized formal power series of the form $f=\sum_{q\in\bfQ}[f(q)]p^q$, where $f(q)\in\overline{\bfF}_p$, $[\cdot]$ is the Teichmüller lift, and the support $\opn{Supp}(f)=\{q\in\bfQ\mid f(q)\neq 0\}$ is a well-ordered subset of $\bfQ$. Krull, Lampert, and Poonen showed that the set of $p$-adic Hahn series, which we denote by $\bfL_p$, forms the spherical completion of $\overline{\bfQ}_p$. In particular, $\bfL_p$ is algebraically closed and complete with respect to the $p$-adic valuation given by $f\mapsto \min\opn{Supp}(f)$.

The field $\bfL_p$ provides a natural setting in which to study transcendental number theory over $p$-adic fields. In this setting, a fundamental question arises:
\begin{question}
    Given a $p$-adic Hahn series $f\in\bfL_p$, how can one determine whether $f$ is a $p$-adic algebraic number, i.e., algebraic over $\bfQ_p$?
\end{question}

Although this question remains open in general, several necessary conditions for a $p$-adic Hahn series to be algebraic over $\bfQ_p$ are known:
\begin{enumerate}
    \item In \cite{lampertAlgebraicPadicExpansions1986} and \cite{poonenMAXIMALLYCOMPLETEFIELDS1993}, Lampert and Poonen proved that if $f\in\bfL_p$ is algebraic over $\bfQ_p$, then
          \begin{enumerate}
              \item there exists an integer $T$ such that $\opn{Supp}(f)\subset \frac{1}{T}\bfZ[1/p]$;
              \item there exists a finite extension $\bfF_q$ of $\bfF_p$ such that $\{f(q)\}_{q\in\bfQ}\subset \bfF_q$.
          \end{enumerate}
          These conditions are also studied quantitatively in \cite{wang2024hyperalgebraicinvariantspadicalgebraic}.
    \item In \cite{lampertAlgebraicPadicExpansions1986}, Lampert also proved that if $f\in\bfL_p$ is algebraic over $\bfQ_p$, then the accumulation points of $\opn{Supp}(f)$ are rational numbers.
    \item In \cite{kedlayaAlgebraicClosurePower2001} and \cite{kedlayaAlgebraicityGeneralizedPower2017b}, Kedlaya gives a necessary and sufficient condition for an equal-characteristic Hahn series in $\bfL_p^\flat\coloneqq\overline{\bfF}_p\pparen{t^\bfQ}$ to lie in the algebraic closure $\overline{\bfF}_p\pparen{t}^{\opn{alg}}$ of $\overline{\bfF}_p\pparen{t}$, phrased in the language of automata theory (cf. \Cref{thm:28713}). As an application, Kedlaya uses Witt vectors to lift this result to the $p$-adic case (cf. \Cref{thm:47}): he shows that the field $\bfC_p$, when viewed as a subfield of $\bfL_p$, coincides with the completion of the set $\Theta\left(\overline{\bfF}_p\pparen{t}^{\opn{alg},\wedge}\right)$, where $\overline{\bfF}_p\pparen{t}^{\opn{alg},\wedge}$ is the $t$-adic completion of $\overline{\bfF}_p\pparen{t}^{\opn{alg}}$ and $\Theta\colon \overline{\bfF}_p\pparen{t^\bfQ}\lto\bfL_p$ is the map $\sum_{q\in\bfQ}f(q)t^q\longmapsto \sum_{q\in\bfQ}[f(q)]p^q$.
\end{enumerate}

Intuitively, Kedlaya's result indicates that the algebraicity of equal-characteristic Hahn series in $\overline{\bfF}_p\pparen{t^\bfQ}$ and the algebraicity of $p$-adic Hahn series in $\bfL_p$ are related to some extent via the map $\Theta$, which leads to the following questions:
\begin{question}\label{qu:29274}\leavevmode
    \begin{enumerate}
        \item Suppose that $f\in\bfL_p^\flat$ is algebraic over $\bfF_p\pparen{t}$. Is $\Theta(f)$ algebraic over $\bfQ_p$?
        \item Suppose that $f\in\bfL_p$ is algebraic over $\bfQ_p$. Is $\Theta^{-1}(f)$ algebraic over $\bfF_p\pparen{t}$?
    \end{enumerate}
\end{question}

In \cite{wang2025padictranscendencesumk1inftyp1pk}, we gave a negative answer to the first question by showing that the $p$-adic Hahn series $\frakA\coloneqq\sum_{k=1}^\infty p^{-1/p^k}$ is transcendental over $\bfQ_p$, whereas its preimage under $\Theta$ is a root of the polynomial $X^p-X-t^{-1}$ over $\bfF_p\pparen{t}$.

\subsection{Sparseness and main theorem}
The key observation in \cite{wang2025padictranscendencesumk1inftyp1pk} is that the support of $\frakA$ is ``sparse'', in the sense that for any nonzero polynomial $P(X)\in\bfQ_p[X]$, the multinomial expansion of $P(\frakA)$ contains orphaned terms that cannot be cancelled by other terms, which forces $P(\frakA)$ to be nonzero. In this article, we generalize this idea to the following combinatorial definitions:
\begin{definition}\label{def:25951}\leavevmode
    \begin{enumerate}
        \item Any rational number $q$ can be uniquely written in the form
              $q=w+\sum_{i=1}^\infty q_i\cdot p^{-i}$, where $w\in\bfZ$, $q_i\in\{0,1,\ldots,p-1\}$ and $q_i\neq p-1$ for infinitely many $i$. We call $\frakN_p(q)\coloneqq\sum_{i=1}^\infty q_i\in \bfN\cup\{\infty\}$ the \textbf{$p$-digit sum} of $q$.
        \item For any set $S$ of rational numbers, define the \textbf{dominant $p$-digit sum} of $S$ to be
        $$\opn{dom}_p(S)\coloneqq\sup\{\frakN_p(q)|q\in S\}\in\bfN\cup\{\infty\},$$ and define the \textbf{$p$-digit dominant part} of $S$ to be the following subset of $S$:
              $$\opn{Dom}_p(S)\coloneqq \{q\in S|\frakN_p(q)=\opn{dom}_p(S)\}.$$
    \end{enumerate}
\end{definition}
\begin{definition}\label{def:15658}
    A set $S\subset [0,1)\cap\bfQ$ is \textbf{sparse} if:
    \begin{enumerate}
        \item $\opn{dom}_p(S)$ is finite;
        \item for infinitely many integers $n\geq 1$, there exist $n$ elements $d_1,\cdots,d_n\in \opn{Dom}_p(S)$ such that
              \begin{enumerate}
                  \item there is no carry in base $p$ when adding $d_1,\cdots,d_n$ together;
                  \item if $e_1,\cdots,e_n\in S$ satisfy that $d_1+\cdots+d_n$ and $e_1+\cdots+e_n$ differ by an integer, then up to a permutation of $\{1,2,\cdots,n\}$, $d_i=e_i$ for every $i=1,2,\cdots,n$.
              \end{enumerate}
    \end{enumerate}
\end{definition}
\begin{remark}
    We give some comments on \Cref{def:15658}:
    \begin{enumerate}
        \item The first condition is closely related to Kedlaya's criterion for the algebraicity of equal-characteristic Hahn series: he shows that if $f\in\bfL_p^\flat$ is algebraic over $\overline{\bfF}_p\pparen{t}$, then $\opn{dom}_p(-T\cdot\opn{Supp}(f))$ is finite for some integer $T\geq 1$.
        \item Condition (2a) indicates that the addition $d_1+\cdots+d_n$, no matter how one parenthesizes it, behaves like the addition in the free commutative monoid $\bigoplus_{\bfZ_{\geq 1}}\bfN$.
        \item Condition (2b) is a combinatorial rigidity condition asserting that the chosen carry-free sum admits a unique decomposition modulo $\bfZ$. This produces the orphaned exponents in the multinomial expansion, which is the key mechanism behind the transcendence proof.
    \end{enumerate}
\end{remark}
The following example illustrates a typical situation in which the sparseness condition is satisfied:
\begin{example}[cf. \Cref{eg:new}]\label{eg:61947}
    Let $A_1,A_2,\cdots$ be a family of pairwise disjoint nonempty subsets of $\bfZ_{\geq 1}$ such that $\sup_i \opn{card}(A_i)<\infty$, and this supremum is attained by infinitely many $i$. Then the set
    $$\left\{\sum_{r\in A_i}p^{-r}\middle\vert i=1,2,\cdots\right\}$$
    is sparse.
\end{example}

The main theorem of this article is the following:
\begin{theorem}[cf. \Cref{thm:23791}]\label{thm:main}
    Let $f\in\bfL_p$ be a $p$-adic Hahn series such that there exists an integer $T\geq 1$ for which $-T\cdot\opn{Supp}(f)$ admits a sparse set $W\neq\{0\}$ of representatives modulo $\bfZ$. Then $f$ is transcendental over $\breve{\bfQ}_p$, and hence over $\bfQ_p$.
\end{theorem}
\begin{remark}
    The sparseness condition does not involve the coefficients of $f$, so it is natural that it does not distinguish algebraicity over $\bfQ_p$ from algebraicity over $\breve{\bfQ}_p$.%
\end{remark}

\subsection{Application: Order type conjecture of $\bfQ_p$-algebraic $p$-adic Hahn series}
In \cite[Section 4]{kedlayaPowerSeriesPAdic2001}, Kedlaya proved that the order type of the support of a $p$-adic Hahn series that is algebraic over $\bfQ_p$ is at most $\omega^\omega$, where $\omega$ is the first infinite ordinal. However, the lack of criteria for the $\bfQ_p$-algebraicity of $p$-adic Hahn series makes it difficult to determine which order types $\leq \omega^\omega$ can be realized as the order type of the support of a $\bfQ_p$-algebraic $p$-adic Hahn series. In the same article, Kedlaya predicts that only finite order types, $\omega$, and $\omega^\omega$ are realizable. The following conjecture was formulated in our previous work, where it was shown to be implied by Kedlaya's prediction (cf. \cite[Proposition 5.2]{wang2025padictranscendencesumk1inftyp1pk}):
\begin{conjecture}[cf. {\cite[Conjecture 5.1]{wang2025padictranscendencesumk1inftyp1pk}}]\label{conj:47376}
    Let $f\in\bfL_p$ be a $p$-adic algebraic number. If $\opn{Supp}(f)$ is bounded, then it must be a finite set.
\end{conjecture}
With the help of \Cref{thm:main}, we prove this conjecture with an additional assumption:
\begin{theorem}[cf. \Cref{coro:11594}]\label{thm:16129}
    Let $f\in\bfL_p$ be a $p$-adic algebraic number with bounded support. If $\opn{Supp}(f)$ has only finitely many accumulation points, then it must be a finite set.
\end{theorem}
\begin{remark}\label{rmk:37294}
    By accumulation points of $\opn{Supp}(f)\subseteq \bfQ$ in \Cref{thm:16129} and other occurrences in this paper, we mean the accumulation point of $\opn{Supp}(f)$ in $\bfR$. In fact, it makes no difference whether one interprets accumulation points in $\bfQ$ or in $\bfR$. By \cite[Theorem 2]{lampertAlgebraicPadicExpansions1986}, the set of all $p$-adic Hahn series $f$ for which the accumulation points of $\opn{Supp}(f)$ are rational numbers is an algebraically closed subfield $F$ of $\bfL_p$. Since $F$ contains $\breve{\bfQ}_p$, it contains all $\breve{\bfQ}_p$-algebraic, and hence all $\bfQ_p$-algebraic, $p$-adic Hahn series. This subtlety was identified during the formalization of this work.
\end{remark}
A direct corollary of this theorem is the following:
\begin{corollary}\label{coro:3344}
    Let $q_1<q_2<\cdots<q_n<\cdots$ be a strictly increasing sequence of rational numbers. If a $p$-adic Hahn series $f=\sum_{i=1}^\infty [f(q_i)]p^{q_i}\in\bfL_p$ is algebraic over $\breve{\bfQ}_p$ (resp. $\bfQ_p$), then one must have $\lim_{i\to\infty}q_i=\infty$.
\end{corollary}

We point out that this corollary actually produces infinitely many counterexamples to the first question in \Cref{qu:29274}. For example, as is mentioned in \cite[Section 1]{kedlayaAlgebraicClosurePower2001}, a prototype of Kedlaya's criterion for the $\overline{\bfF}_p$-algebraicity of Hahn series in $\bfL_p^\flat$ is the following result of Huang, which was independently discovered by Ştefănescu:
\begin{proposition}[{cf. \cite{HuangThesis1968,stefanescuMeromorphicFormalPower1983}}]
    Let $f=\sum_{i=1}^\infty f(i)\cdot t^{-1/p^i}\in\bfL_p^\flat$ be a Hahn series. Then the following are equivalent:
    \begin{enumerate}
        \item the series $f$ is algebraic over $\bfF_p\pparen{t}$;
        \item the series $f$ is algebraic over $\overline{\bfF}_p\pparen{t}$;
        \item the sequence $\{f(i)\}_{i\geq 1}$ is eventually periodic.
    \end{enumerate}
\end{proposition}

The $p$-adic analogue of this result is immediate from \Cref{coro:3344}:
\begin{proposition}\label{prop:31772}
    Let $f=\sum_{i=1}^\infty [f(i)]\cdot p^{-1/p^i}\in\bfL_p$ be a $p$-adic Hahn series. Then the following are equivalent:
    \begin{enumerate}
        \item the series $f$ is algebraic over $\bfQ_p$;
        \item the series $f$ is algebraic over $\breve{\bfQ}_p$;
        \item $f(i)=0$ for all but finitely many $i$.
    \end{enumerate}
\end{proposition}

We hope the approach in this article can be further developed to give a full answer to \Cref{conj:47376}.

\subsection{Formalization in Lean}
    Given the highly combinatorial nature of the sparseness condition and the transcendence proof, we formalize all results of this article in the Lean theorem prover (cf. \cite{demouraLeanTheoremProver2015,mouraLean4TheoremProver2021}, v 4.31.0), a proof assistant based on dependent type theory, building over Mathlib (cf. \cite{mathlib2020}, \href{https://github.com/leanprover-community/mathlib4/tree/fabf563a7c95a166b8d7b6efca11c8b4dc9d911f}{Commit abf563}). This formalization was carried out with the help of the agentic auto-formalization system Archon (cf. \cite{ju2026automatedconjectureresolutionformal}), developed by the AI4Math team at BICMR, Peking University.
\begin{remark}
    The formalization of all contents of \Cref{sec:18373}, \Cref{sec:sparse}, \Cref{sec:tscaled} and \Cref{sec:main} is completely \lean{sorry}-free. For the formalization of \Cref{sec:55224} and the consequential \Cref{coro:3344} and \Cref{prop:31772}, we use two results of Kedlaya (cf. \Cref{thm:28713} and \Cref{thm:47}) as black boxes, for they require a significant amount of work to formalize and are far beyond the scope of this article. These are the only two \lean{sorry}-s in the formalization project.
\end{remark}
The formalization is available at \url{https://github.com/YijunYuan/FormalizedSparse}, and we refer the reader to Appendix \ref{sec:formal} for a detailed discussion on the formalization process.

\subsection*{Acknowledgements}
The authors would like to thank Wanying He and Jiedong Jiang for their assistance with using Archon. The research is partially supported by the National Key R\&D Program of China (Grant No. 2024YFA1014000).

\section{Preliminaries on Hahn series}\label{sec:18373}
To keep this article self-contained, we briefly recall some basic facts about Hahn series.%

\begin{definition}[{\cite[Section 3]{poonenMAXIMALLYCOMPLETEFIELDS1993}}]\label{def:61947}
    Let $R$ be a commutative ring and $G$ be an ordered group.
    \begin{enumerate}
        \item For any $f\in\opn{Hom}_{\opn{Set}}(G,R)$, we define the \textbf{support} of $f$ to be $$\opn{Supp}(f)=\{g\in G\colon f(g)\neq 0\}.$$
        \item Define the set of \textbf{Hahn series} over $R$ with value group $G$ to be
              $$R\pparen{G}\coloneqq \{f\in \opn{Hom}_{\opn{Set}}(G,R)\colon \opn{Supp}(f) \text{ is well-ordered}\}.$$
              By introducing a formal variable $t$, elements in $R\pparen{G}$ will also be written as $\sum_{g\in G}r_gt^g$, where $r_g\in R$ for all $g\in G$.
    \end{enumerate}
\end{definition}
\begin{proposition}[{\cite[Lemma 1,Corollary 2]{poonenMAXIMALLYCOMPLETEFIELDS1993}}]\label{prop:40704}
    Let $R$ be a commutative ring and $G$ be an ordered group.
    \begin{enumerate}
        \item With identity $1\cdot t^0$ and addition as well as multiplication given by
              $$\sum_{g\in G}a_g t^g+\sum_{g\in G}b_g t^g\coloneqq\sum_{g\in G}(a_g+b_g)t^g,\ \sum_{g\in G}a_g t^g\cdot\sum_{g\in G}b_g t^g\coloneqq\sum_{g\in G}\left(\sum_{h\in G}a_h b_{g-h}\right)t^g,$$
              $R\pparen{G}$ forms a commutative ring.
        \item If $R$ is a field, then so is $R\pparen{G}$. Moreover, with the map
              $$v\colon R\pparen{G}\lto G\cup\{\infty\},\ f\longmapsto \begin{cases}\min\opn{Supp}(f),& \text{ if }f\neq 0\\\infty,& \text{ if }f=0\end{cases},$$
              $R\pparen{G}$ becomes a valued field with value group $G$ and residue field $R$.
    \end{enumerate}
\end{proposition}
Since $\opn{char}R\pparen{G}=\opn{char}R$, we call $R\pparen{G}$ the \textbf{equal-characteristic field of Hahn series} over $R$ with value group $G$, also denoted by $R\pparen{t^G}$ with respect to the formal variable $t$.

\begin{proposition}[{\cite[Proposition 3, Corollary 3, Proposition 5]{poonenMAXIMALLYCOMPLETEFIELDS1993}}]\label{prop:49318}
    Let $k$ be a perfect field of characteristic $p$ and $G$ be an ordered group containing $\bbZ$ as a subgroup. In addition, let
    $$\calN\coloneqq \left\{\sum_{g\in G}r_g t^g\in W(k)\pparen{t^G}\colon \text{ for every } g\in G,\ \sum_{n\in\bbZ}r_{g+n}p^n=0\right\},$$
    where $W(k)$ is the ring of Witt vectors of $k$. We call elements in $\calN$ the \textbf{null series} of $W(k)\pparen{t^G}$. Then
    \begin{enumerate}
        \item $\calN$ is a maximal ideal of $W(k)\pparen{t^G}$, which makes $W(k)\pparen{p^G}\coloneqq W(k)\pparen{t^G}/\calN$ a field\footnote{Informally, $W(k)\pparen{p^G}$ is obtained by replacing the formal variable $t$ in elements of $W(k)\pparen{t^G}$ by the prime $p$.}, called the \textbf{field of $p$-adic Hahn series}.
        \item Every element $x$ in $W(k)\pparen{p^G}$ can be uniquely written as $$x=\sum_{g\in G}[r_g]p^g,$$
              where $r_g\in k$ for all $g\in G$ and $[\cdot]\colon k\lto W(k)$ is the Teichmüller lift. We call this the \textbf{standard expansion} of the element $x$.
        \item For $f=\sum_{g\in G}[r_g]p^g\in W(k)\pparen{p^G}$, define the \textbf{support} of $f$ to be
              $$\opn{Supp}(f)=\{g\in G\colon r_g\neq 0\}.$$
              Then the map
              $$v\colon W(k)\pparen{G}/\calN\lto G\cup\{\infty\},\ f\mapsto \begin{cases}
                      \min\opn{Supp}(f), & \text{ if }f\neq 0 \\\infty,& \text{ if }f=0
                  \end{cases}$$
              makes $W(k)\pparen{G}/\calN$ a mixed-characteristic valued field with value group $G$ and residue field $k$.
    \end{enumerate}
\end{proposition}

The most fundamental property of the field of Hahn series is the following:
\begin{theorem}[cf. {\cite[Theorem 1, Corollary 4, Corollary 6]{poonenMAXIMALLYCOMPLETEFIELDS1993}}]
    Let $F$ be an equal-characteristic (resp. mixed-characteristic) valued field with divisible value group $G$ and algebraically closed residue field $k$. Then the equal-characteristic (resp. $p$-adic) field of Hahn series $k\pparen{t^G}$ (resp. $W(k)\pparen{p^G}$) is the unique (up to isomorphism of valued fields) minimal spherically complete extension of $F$. Moreover, it is algebraically closed and complete.
\end{theorem}

The following are the fields of Hahn series used in this article:
\begin{example}\label{eg:44541}
    Let $F=\overline{\bfF}_p\pparen{t}^{\opn{alg}}$ (resp. $\breve{\bfQ}_p^{\opn{alg}}$), which has value group $\bfQ$ and residue field $\overline{\bfF}_p$. Then the field of equal-characteristic (resp. $p$-adic) Hahn series $\bfL_p^\flat\coloneqq\overline{\bfF}_p\pparen{t^\bfQ}$ (resp. $\bfL_p\coloneqq\breve{\bfZ}_p\pparen{p^\bfQ}=W(\overline{\bfF}_p)\pparen{p^\bfQ}$) is the spherical completion of $F$ with the same residue field and value group, and is algebraically closed and complete.
\end{example}
For brevity, we call elements of $\bfL_p$ $p$-adic Hahn series without specifying the residue field and value group.

\section{Sparseness and $(c,n)$-sparseness}\label{sec:sparse}
To give a rigorous and workable formulation of the sparseness condition in \Cref{def:15658}, we introduce several combinatorial notions and auxiliary functions related to base-$p$ expansions of rational numbers in $[0,1)$.
\begin{definition}\leavevmode
    \begin{enumerate}
        \item For $\uld\in \bigoplus_{\bfZ_{\geq 1}}\bfN$, let $\Psi(\uld)\coloneqq \sum_{i=1}^\infty d_i\in\bfN$ and $\normd\coloneqq \sum_{i=1}^\infty d_i p^{-i}\in\bfQ_{\geq 0}$.
        \item Let $\bbP\coloneqq \bigoplus_{\bfZ_{\geq 1}}\{0,1,\cdots,p-1\}\subsetneq \bigoplus_{\bfZ_{\geq 1}}\bfN$.
        \item For $\uld\in \bigoplus_{\bfZ_{\geq 1}}\bfN$ and $i\in\bfZ_{\geq 1}$, denote by $\uld_i$ the $i$-th component of $\uld$.
    \end{enumerate}
\end{definition}
\begin{lemma}
    For any $\uld\in\bigoplus_{\bfZ_{\geq 1}}\bfN$, there exists a unique element $\tau(\uld)\in\bbP$ such that $\normd-\lVert\tau(\uld)\rVert\in\bfZ$.
\end{lemma}
\begin{proof}
    If one writes $\normd$ as a decimal expansion\footnote{We do not allow infinite strings of $(p-1)$ in the decimal expansion of $\normd$ to ensure the uniqueness of $\tau(\uld)$.
    } in base $p$:
    $$\normd=w.d_1\cdots d_n\cdots,$$
    where $w\in\bfZ$ and $d_i\in\{0,1,\cdots,p-1\}$ for any $i$, then $\tau(\uld)=0.d_1\cdots d_n\cdots$. The uniqueness is trivial.
\end{proof}
We collect several properties of these concepts in the following lemma:
\begin{lemma}\label{lem:18813}
    Let $\uld,\underline{e}\in\bigoplus_{\bfZ_{\geq 1}}\bfN$.
    \begin{lemenum}
        \item\label{it:28363} The maps $\Psi(\cdot)$ and $\lVert \cdot\rVert$ are additive. Moreover, $\norm{\cdot}$ is injective when restricted to $\bbP$.
        \item\label{it:52935} One has $\tau(\uld)=\tau(\underline{e})$ if and only if $\normd-\lVert \underline{e}\rVert\in\bfZ$.
        \item $\uld\in\bbP$ if and only if $\tau(\uld)=\uld$.
        \item\label{lem:17816} One has $\Psi(\tau(\uld))\leq \Psi(\uld)$, with equality if and only if $\uld\in\bbP$.
    \end{lemenum}
\end{lemma}
\begin{proof}
    The first three statements are straightforward. We only prove the last statement.

    For any $\uld\in\bigoplus_{\bfZ_{\geq 1}}\bfN$, let
    $$N(\uld)\coloneqq\begin{cases}
            0,                   & \text{ if }\uld\in\bbP; \\
            \max\{i|d_i\geq p\}, & \text{ otherwise}.
        \end{cases}$$
    By a descent argument on $N(\uld)$, it suffices to show that if $N(\underline{d})\geq 1$, then there exists $\underline{e}\in\bigoplus_{\bfZ_{\geq 1}}\bfN$ such that $\normd-\lVert \underline{e}\rVert\in\bfZ$, $\Psi(\underline{e})<\Psi(\uld)$, and $N(\underline{e})<N(\underline{d})$.

    Writing $d_{N(\uld)}=p\cdot r + s$ with $r\in\bfN$ and $s\in\{0,1,\cdots,p-1\}$, we take $\underline{e}\coloneqq (d_0,\cdots,d_{N(\uld)-2},d_{N(\uld)-1}+r,s,0,\cdots)$. Then
    $$\Psi(\uld)-\Psi(\underline{e})=d_{N(\uld)}-r-s= (p-1)\cdot r>0$$
    and $\normd-\lVert \underline{e}\rVert=0$. The result follows.
\end{proof}
\begin{remark}
    For a rational number $q$ in $[0,1)$ with a finite-length decimal expansion in base $p$, the preimage $\underline{q}\in\bbP$ of $q$ under the map $\norm{\cdot}$ extracts the digits of $q$ in base $p$, and $\Psi(\underline{q})$ is the $p$-digit sum of $q$.
\end{remark}
\begin{definition}\label{def:38144}
    Let $p$ be a prime number. Let $c, n \geq 1$ be integers. A subset $S\subset\bbP$ is \textbf{$(c,n)$-sparse} if
    \begin{enumerate}
        \item there exists $c\geq 1$ such that $\Psi(\uld)\leq c$ for every $\uld\in S$;
        \item there exist $n$ (not necessarily distinct) elements $\uld^{(1)}, \uld^{(2)}, \cdots, \uld^{(n)}\in S$ such that
              \begin{enumerate}
                  \item $\Psi\left(\uld^{(i)}\right)=c$ for $i=1,2,\cdots,n$ and $\sum_{i=1}^n \uld^{(i)}_j< p$ for $j\in\bfN$ (that is, $\sum_{i=1}^n \uld^{(i)}\in\bbP$);
                  \item if $\underline{e}^{(1)}, \underline{e}^{(2)}, \cdots, \underline{e}^{(n)}\in S$ satisfy $\left\lVert\sum_{i=1}^n \uld^{(i)}\right\rVert-\left\lVert\sum_{i=1}^n \underline{e}^{(i)}\right\rVert\in\bfZ$, then up to a permutation of $\{1,2,\cdots,n\}$, $\uld^{(i)}=\underline{e}^{(i)}$ for every $i=1,2,\cdots,n$.
              \end{enumerate}
    \end{enumerate}
\end{definition}
The following lemma reformulates the sparseness condition in \Cref{def:15658} in terms of the $(c,n)$-sparseness condition in \Cref{def:38144}.
\begin{lemma}\label{lem:54409}
    A subset $W\neq\{0\}$ of $[0,1)\cap\bfQ$ is sparse in the sense of \Cref{def:15658} if and only if there exists a subset $S$ of $\bbP$ such that $W=\norm{S}$ and there exists an integer $c\geq 1$ such that $S$ is $(c,n)$-sparse for infinitely many integers $n\geq 1$.
\end{lemma}
\begin{proof}
    Suppose that $W\subset [0,1)\cap\bfQ$ is sparse. Then the $p$-digit sum of every element in $W$ is finite and bounded by $\opn{dom}_p(W)$. Thus we may take
    $$S\coloneqq \left\{(q_i)_{i\in\bfZ_{\geq 1}}|q=\sum_{i=0}^\infty q_i p^{-i}\in W,\ q_i\in\{0,\cdots,p-1\}\right\}.$$
    Then $S$ is $(\opn{dom}_p(W),n)$-sparse for infinitely many integers $n\geq 1$.

    The converse direction follows by a similar argument.
\end{proof}

We give a concrete example of a sparse set, which is the prototype of the situation in which the sparseness condition is satisfied:
\begin{example}\label{eg:new}
    Let $\underline{A}\coloneqq(A_i)_{i\geq 1}$ be a family of pairwise disjoint nonempty subsets of $\bfZ_{\geq 1}$ such that $\sup_i |A_i|<\infty$ and this supremum is attained by infinitely many $i$. Then the set
    $$M(\underline{A})\coloneqq\left\{\sum_{r\in A_i}p^{-r}\middle\vert i=1,2,\cdots\right\}\subset [0,1)\cap\bfQ$$
    is sparse.
\end{example}
\begin{proof}
    For any $i$, one has $\frakN_p(\sum_{r\in A_i}p^{-r})=|A_i|$, and consequently $\opn{dom}_p(M(\underline{A}))=\sup_i |A_i|<\infty$.  Since the supremum is attained by infinitely many $i$, the set $\opn{Dom}_p(M(\underline{A}))$ is infinite:
    $$\opn{Dom}_p(M(\underline{A}))=\left\{\mathbf{d}_j\coloneqq \sum_{r\in A_{i_j}}p^{-r}\middle\vert j=1,2,\cdots\right\},$$
    with $\mathbf{d}_k\neq \mathbf{d}_l$ for any $k\neq l$.%

    Take $e_1,\cdots,e_n\in M(\underline{A})$ such that $\sum_{j=1}^n e_j-\sum_{j=1}^n \mathbf{d}_j\in\bfZ$. Since $M(\underline{A})\subset \norm{\bbP}$, we may write $e_j=\norm{\underline{e}^{(j)}}$ and $\mathbf{d}_j=\norm{\underline{\mathbf{d}}^{(j)}}$ with $\underline{e}^{(j)}, \underline{\mathbf{d}}^{(j)}\in \bbP$ for every $j=1,\cdots,n$.  Then one has $\sum_{j=1}^n \mathbf{d}_j=\sum_{r\in\bigsqcup_{j=1}^n A_{i_j}}p^{-r}$, and consequently $\sum_{j=1}^n \underline{\mathbf{d}}^{(j)}\in\bbP$.
    In particular,
    \begin{equation*}
        n\cdot\opn{dom}_p(M(\underline{A}))=\frakN_p\left(\sum_{r\in\bigsqcup_{j=1}^n A_{i_j}}p^{-r}\right)=\frakN_p\left(\sum_{j=1}^n \mathbf{d}_j\right)=\Psi\left(\sum_{j=1}^n \underline{\mathbf{d}}^{(j)}\right).
    \end{equation*}
    Since
    \begin{align*}
        \Psi\left(\sum_{j=1}^n \underline{\mathbf{d}}^{(j)}\right)=\Psi\left(\tau\left(\sum_{j=1}^n \underline{\mathbf{d}}^{(j)}\right)\right) & =\Psi\left(\tau\left(\sum_{j=1}^n \underline{e}^{(j)}\right)\right)                                                                                 \\
                                                                                                                                                   & \leq \Psi\left(\sum_{j=1}^n \underline{e}^{(j)}\right)=\sum_{j=1}^n \Psi\left(\underline{e}^{(j)}\right)\leq n\cdot\opn{dom}_p(M(\underline{A})),
    \end{align*}
    one concludes that
    $\Psi\left(\underline{e}^{(j)}\right)=\opn{dom}_p(M(\underline{A}))$ for every $j=1,\cdots,n$ and $\sum_{j=1}^n \underline{e}^{(j)}\in\bbP$. Consequently, one has $\sum_{j=1}^n \underline{e}^{(j)}=\sum_{j=1}^n \underline{\mathbf{d}}^{(j)}$, implying that $\sum_{j=1}^n e_j=\sum_{j=1}^n \mathbf{d}_j$.

    Note that there is no duplication among $e_1,\cdots,e_n$: if not, then there exists a coordinate of $\sum_{j=1}^n \underline{e}^{(j)}$ that is at least $2$, contradicting the fact that every coordinate of $\sum_{j=1}^n \underline{\mathbf{d}}^{(j)}$ is $0$ or $1$. If one writes $e_j=\sum_{r\in A_{k_j}}p^{-r}$ for every $j=1,\cdots,n$, then
    $$\sum_{r\in\bigsqcup_{j=1}^n A_{k_j}}p^{-r}=\sum_{j=1}^n e_j=\sum_{j=1}^n \mathbf{d}_j=\sum_{r\in\bigsqcup_{j=1}^n A_{i_j}}p^{-r},$$
    implying that $\bigsqcup_{j=1}^n A_{k_j}=\bigsqcup_{j=1}^n A_{i_j}$. This forces $A_{k_j}=A_{i_j}$ and consequently $e_j=\mathbf{d}_j$ for every $j=1,\cdots,n$, up to a permutation of $\{1,\cdots,n\}$.
\end{proof}

We end this section with the following technical lemma, which will be used in the proof of \Cref{thm:main} to extract the orphaned exponents in the multinomial expansion.
\begin{lemma}\label{lem:39609}
    Let $S\subset\bbP$ be a $(c,n)$-sparse subset for some integers $c, n\geq 1$. Let $\underline{d}^{(1)},\cdots,\underline{d}^{(n)}$ be elements of $S$ that satisfy the condition (2) of \Cref{def:38144} and set
    $$\phi_0\colon S\lto\bfN,\ \underline{d}\longmapsto \opn{card}\left(\left\{i\vert 1\leq i\leq n,\ \uld=\uld^{(i)}\right\}\right).$$
    Then $\phi_0$ is the unique function $\phi\colon S\to\bfN$ such that $\sum_{\uld\in S}\phi(\uld)\leq n$ and
    $$\sum_{\uld\in S}\normd\cdot\phi(\uld)\equiv \sum_{\uld\in S}\normd\cdot\phi_0(\uld)\pmod{\bfZ}.$$
\end{lemma}
\begin{proof}
    Suppose that $\phi_1\colon S\to\bfN$ is another function such that $\sum_{\uld\in S}\phi_1(\uld)\leq n$ and
    $$\sum_{\uld\in S}\normd\cdot\phi_1(\uld)\equiv \sum_{\uld\in S}\normd\cdot\phi_0(\uld)\pmod{\bfZ}.$$
    \begin{enumerate}
        \item If $\sum_{\uld\in S}\phi_1(\uld)< n$, then by \Cref{lem:17816} we have
              $$\Psi\left(\tau\left(\sum_{\uld\in S}\uld\cdot\phi_1(\uld)\right)\right)\leq\Psi\left(\sum_{\uld\in S}\uld\cdot\phi_1(\uld)\right)=\sum_{\uld\in S}\Psi(\uld)\cdot\phi_1(\uld)<c\cdot n.$$
              On the other hand, one has $\sum_{j=1}^n\uld_i^{(j)}<p$ for any $i$ by the sparseness condition (2a) of \Cref{def:38144}. Hence
              $$\sum_{\uld\in S}\uld\cdot\phi_0(\uld)=\left(\sum_{j=1}^n\uld_i^{(j)}\right)_{i\in\bfZ_{\geq 0}}=\tau\left(\sum_{\uld\in S}\uld\cdot\phi_0(\uld)\right)$$
              and consequently
              $$\Psi\left(\tau\left(\sum_{\uld\in S}\uld\cdot\phi_0(\uld)\right)\right)=\sum_{i=0}^\infty \left(\sum_{j=1}^n\uld_i^{(j)}\right)=\sum_{j=1}^n \Psi(\uld^{(j)})=n\cdot c.$$
              Thus, we have $\tau\left(\sum_{\uld\in S}\phi_1(\uld)\cdot\uld\right)\neq \tau\left(\sum_{\uld\in S}\uld\cdot\phi_0(\uld)\right)$, which leads to a contradiction by \Cref{it:52935}.
        \item If $\sum_{\uld\in S}\phi_1(\uld)= n$, then we take $\underline{e}^{(1)}, \underline{e}^{(2)}, \cdots, \underline{e}^{(n)}\in S$ such that for every $\uld\in S$, $\uld$ appears exactly $\phi_1(\uld)$ times in the list $\underline{e}^{(1)}, \underline{e}^{(2)}, \cdots, \underline{e}^{(n)}$. Then
              $$\sum_{\uld\in S}\normd\cdot\phi_1(\uld)-\sum_{\uld\in S}\normd\cdot\phi_0(\uld)=\left\lVert\sum_{\uld\in S}\uld\cdot\phi_1(\uld)\right\rVert-\left\lVert\sum_{\uld\in S}\uld\cdot\phi_0(\uld)\right\rVert=\left\lVert\sum_{i=1}^n\underline{e}^{(i)}\right\rVert-\left\lVert\sum_{i=1}^n\uld^{(i)}\right\rVert\in\bfZ.$$
              By the sparseness condition (2b) of \Cref{def:38144}, up to a permutation of $\{1,2,\cdots,n\}$, $\uld^{(i)}=\underline{e}^{(i)}$ for any $i=1,2,\cdots,n$. This implies that $\phi_1=\phi_0$.
    \end{enumerate}
\end{proof}
\section{$T$-scaled realization of $\bfL_p$}\label{sec:tscaled}
To prove \Cref{thm:main}, one must group the terms of a $p$-adic Hahn series $f=\sum_{q\in\bfQ}[f(q)]p^q$ in $\bfL_p$  by the residue of the exponent modulo $\frac{1}{T}\bfZ$ for some integer $T\geq 1$. This works when $T=1$, since a direct computation shows that the element
$$\sum_{q\in\opn{Supp}(f)/\bfZ}\left(\sum_{w\in\bfZ}f(q+w)p^w\right)t^q \in \breve{\bfZ}_p\pparen{t^\bfQ}$$
is a preimage of $f$ under the natural projection $\breve{\bfZ}_p\pparen{t^\bfQ}\to \bfL_p$, where
$$\opn{Supp}(f)/\bfZ\coloneqq\left\{\inf\left(\opn{Supp}(f)\cap (q+\bfZ)\right)|q\in\opn{Supp}(f)\right\}\subset\bfQ$$
is a well-ordered set of representatives of $\opn{Supp}(f)$ modulo $\bfZ$. However, when $T>1$, the same construction fails, because the element $\sum_{w\in\frac{1}{T}\bfZ}f(q+w)p^w$ does not necessarily lie in $\breve{\bfZ}_p$, so a direct analogue of the above construction is not well-defined\footnote{This issue was detected during the formalization of the proof of the main theorem in Lean 4.}. To resolve this, we enlarge the ring $\breve{\bfZ}_p=W(\overline{\bfF}_p)$ to include $p^{1/T}$, and realize $\bfL_p$ as a quotient of $\breve{\bfZ}_p[p^{1/T}]\pparen{t^\bfQ}$ by a suitable ideal. This is the main content of this section.

From now on, $T$ will be a positive integer. We set $\breve{\bfZ}_{p,T}\coloneqq \breve{\bfZ}_p[p^{1/T}]$ and $\breve{\bfQ}_{p,T}\coloneqq \breve{\bfQ}_p(p^{1/T})$.
\begin{lemma}\label{lem:61648}
    One has $\left[\breve{\bfQ}_{p,T}\colon\breve{\bfQ}_p\right]=T$. In particular, $1, p^{1/T},\cdots, p^{(T-1)/T}$ form a basis of $\breve{\bfQ}_{p,T}$ over $\breve{\bfQ}_p$.
\end{lemma}
\begin{proof}
    Since $\breve{\bfQ}_p$ is a discrete valuation field and $p^{1/T}$ is a root of the Eisenstein polynomial $X^T-p$, the result follows from the Eisenstein criterion.
\end{proof}
\begin{remark}
    $\breve{\bfZ}_p\pparen{t^\bfQ}$ is a subfield of $\breve{\bfZ}_{p,T}\pparen{t^\bfQ}$ via the natural inclusion $\breve{\bfZ}_p\subseteq \breve{\bfZ}_{p,T}$.
\end{remark}
\begin{definition}
    Let $\calN_T$ be the set of elements $\sum_{q\in\bfQ}c_q t^q\in \breve{\bfZ}_{p,T}\pparen{t^\bfQ}$ such that for any $q\in\bfQ$,
    $$\sum_{n\in\bfZ}c_{q+\frac{n}{T}}p^{\frac{n}{T}}=0.$$
    We call these elements the \textbf{$T$-null-series} in $\breve{\bfZ}_{p,T}\pparen{t^\bfQ}$.
\end{definition}
\begin{remark}
    When $T=1$, $\calN_T$ coincides with the ideal $\calN$ of null-series in $\breve{\bfZ}_p\pparen{t^\bfQ}$ defined in \Cref{prop:49318}.
\end{remark}
\begin{lemma}\leavevmode
    \begin{enumerate}
        \item $\calN_T$ is an ideal of $\breve{\bfZ}_{p,T}\pparen{t^\bfQ}$.
        \item For every element $f$ of $\breve{\bfZ}_{p,T}\pparen{t^\bfQ}$, there exists a unique element $g=\sum_{q\in\bfQ}[g(q)]t^q$ in $\breve{\bfZ}_{p,T}\pparen{t^\bfQ}$ such that $f-g\in\calN_T$.
        \item $\calN_T$ is a maximal ideal of $\breve{\bfZ}_{p,T}\pparen{t^\bfQ}$, so that $\breve{\bfZ}_{p,T}\pparen{t^\bfQ}/\calN_T$ is a field.
    \end{enumerate}
\end{lemma}
\begin{proof}
    These three statements are generalizations of \cite[Proposition 3]{poonenMAXIMALLYCOMPLETEFIELDS1993}, \cite[Proposition 4]{poonenMAXIMALLYCOMPLETEFIELDS1993} and \cite[Corollary 3]{poonenMAXIMALLYCOMPLETEFIELDS1993} respectively, and the proofs are essentially the same.
\end{proof}
\begin{remark}
    By (2) of this lemma, we will formally write elements of $\breve{\bfZ}_{p,T}\pparen{t^\bfQ}/\calN_T$ as $\sum_{q\in\bfQ}[g(q)]p^q$.
\end{remark}
\begin{lemma}\label{lem:14623}
    One has $\calN_T\cap \breve{\bfZ}_p\pparen{t^\bfQ}=\calN$.
\end{lemma}
\begin{proof}
    Let $f=\sum_{q\in\bfQ}[f(q)]t^q\in\calN$, then for any $q\in\bfQ$,
    \begin{equation}\label{eq:64661}
        \sum_{n\in\bfZ}[f(q+n)]p^n=0.
    \end{equation}
    Notice that for any $q\in\bfQ$, one has
    \begin{align*}
        \sum_{n\in\bfZ}\left[f\left(q+\frac{n}{T}\right)\right]p^{\frac{n}{T}}= & \sum_{u=0}^{T-1}\sum_{\substack{n\in\bfZ                                                                  \\n\equiv u \pmod{T}}}\left[f\left(q+\frac{n}{T}\right)\right]p^{\frac{n}{T}}\\
        =                                                                       & \sum_{u=0}^{T-1}p^{\frac{u}{T}}\sum_{n\in\bfZ}\left[f\left(\left(q+\frac{u}{T}\right)+n\right)\right]p^n.
    \end{align*}
    By applying \Cref{eq:64661} to $q+\frac{u}{T}$ for $u=0,1,\cdots,T-1$, one has $\sum_{n\in\bfZ}\left[f\left(q+\frac{n}{T}\right)\right]p^{\frac{n}{T}}=0$. As a result, $f\in\calN_T$.

    Conversely, let $g=\sum_{q\in\bfQ}[g(q)]t^q\in\calN_T$. Then for any $q\in\bfQ$, one has
    $$\sum_{n\in\bfZ}\left[g\left(q+\frac{n}{T}\right)\right]p^{\frac{n}{T}}=\sum_{u=0}^{T-1}p^{\frac{u}{T}}\sum_{n\in\bfZ}\left[g\left(\left(q+\frac{u}{T}\right)+n\right)\right]p^n=0.$$
    By \Cref{lem:61648}, for any $u=0,1,\cdots,T-1$, one has $\sum_{n\in\bfZ}\left[g\left(\left(q+\frac{u}{T}\right)+n\right)\right]p^n=0$. The result follows from taking $u=0$.
\end{proof}
\begin{lemma}\label{lem:21624}
    One has $\breve{\bfZ}_p\pparen{t^\bfQ}+\calN_T=\breve{\bfZ}_{p,T}\pparen{t^\bfQ}$.
\end{lemma}
\begin{proof}
    Take $f=\sum_{q\in\bfQ}c_q t^q\in \breve{\bfZ}_{p,T}\pparen{t^\bfQ}$, where
    $$c_q=\sum_{n=0}^\infty [c_{q,n}]p^{\frac{n}{T}}\in \breve{\bfZ}_{p,T}.$$
    Then one has $f=f_0+p^{\frac{1}{T}}f_1+\cdots+p^{\frac{T-1}{T}}f_{T-1}$, where
    $$f_u=\sum_{t\in\bfQ}\left(\sum_{\substack{n\in\bfN\\n\equiv u\pmod{T}}} [c_{q,n}]p^{\frac{n-u}{T}}\right)t^q \in \breve{\bfZ}_p\pparen{t^\bfQ},\ u=0,1,\cdots,T-1.$$
    It suffices to show that $p^{\frac{u}{T}}f_u$ belongs to $\breve{\bfZ}_p\pparen{t^\bfQ}+\calN_T$ for $u=0,1,\cdots,T-1$. This is clear if we write $f_u$ as $f_u=(p^{\frac{u}{T}}\cdot t^0-t^{\frac{u}{T}})f_u+t^{\frac{u}{T}}f_u$, where $t^{\frac{u}{T}}f_u\in \breve{\bfZ}_p\pparen{t^\bfQ}$ and $(p^{\frac{u}{T}}\cdot t^0-t^{\frac{u}{T}})f_u$ is a $T$-null-series.
\end{proof}
\begin{proposition}\label{prop:18592}
    The inclusion $\iota\colon \breve{\bfZ}_p\pparen{t^\bfQ}\hookrightarrow \breve{\bfZ}_{p,T}\pparen{t^\bfQ}$ induces an isomorphism of fields
    $$\sigma\colon\bfL_p\xlongrightarrow{\cong}\breve{\bfZ}_{p,T}\pparen{t^\bfQ}/\calN_T,\ \sum_{q\in\bfQ}[f(q)]p^q\longmapsto \sum_{q\in\bfQ}[f(q)]p^q,$$
    i.e., the following diagram commutes:
    \begin{equation}\label{eq:28959}
        \begin{tikzcd}
            \breve{\bfZ}_p\pparen{t^\bfQ}\ar[r,hook,"\iota"]\ar[d, ->>,"\pi"]&\breve{\bfZ}_{p,T}\pparen{t^\bfQ}\ar[d, ->>,"\pi_T"]\\
            \bfL_p\ar[r,"\cong","\sigma"']&\breve{\bfZ}_{p,T}\pparen{t^\bfQ}/\calN_T
        \end{tikzcd}.
    \end{equation}
\end{proposition}
\begin{proof}
    This is a direct consequence of the following standard fact in commutative algebra: given domains $R_1\subset R_2$ and an ideal $I$ of $R_2$, if $R_1+I=R_2$, then $R_1/(I\cap R_1)\cong R_2/I$, where the isomorphism is induced by the inclusion $R_1\subset R_2$.
\end{proof}

As an application, the $T$-scaled realization in \Cref{prop:18592} gives the desired grouping of terms of a $p$-adic Hahn series by residue modulo $\frac{1}{T}\bfZ$ in the case $T>1$, and provides an explicit lift of $\sigma(f)$ that will serve as the starting point for the proof of the main theorem. We record this lift below.

Let $f=\sum_{q\in\bfQ}[f(q)]p^q\in\bfL_p$ and let $W\subset [0,1)\cap\bfQ$ be a set of representatives of $-T\cdot\opn{Supp}(f)$ modulo $\bfZ$. Intuitively, the field $\breve{\bfZ}_{p,T}\pparen{t^\bfQ}/\calN_T$ provides a suitable setting to write $f$ as $\sum_{q\in W}C_{-\frac{q}{T}} p^{-\frac{q}{T}}$, where $C_{-\frac{q}{T}}\in \breve{\bfZ}_{p,T}$ for every $q\in W$. To be more precise, for every $q\in W$, we set
    $$n_q\coloneqq -\frac{q}{T}-\inf\left(\opn{Supp}(f)\cap\left(-\frac{q}{T}+\frac{1}{T}\bfZ\right)\right)\in \frac{1}{T}\bfZ.$$
    This element is well-defined. Indeed, the set $\opn{Supp}(f)\cap\left(-\frac{q}{T}+\frac{1}{T}\bfZ\right)$ is nonempty by the assumption that $-\frac{1}{T}W$ is a set of representatives of $\opn{Supp}(f)$ modulo $\frac{1}{T}\bfZ$. Moreover, it is a subset of the well-ordered set $\opn{Supp}(f)$, and hence has a minimum element $n_q$. Then the set
    $$\wtilde{W}\coloneqq \left\{-\frac{q}{T}-n_{q}\middle\vert q\in W\right\}\subseteq \opn{Supp}(f)$$
    is a well-ordered set of representatives of $\opn{Supp}(f)$ modulo $\frac{1}{T}\bfZ$, and is in bijection with $W$ via the map
    \begin{equation}\label{eq:mu-bijection}
        \mu_W\colon W\lto \wtilde{W},\ q\longmapsto -\frac{q}{T}-n_q=\inf\left(\opn{Supp}(f)\cap\left(-\frac{q}{T}+\frac{1}{T}\bfZ\right)\right).
    \end{equation}

    For any $w\in \wtilde{W}$, set $C_w\coloneqq\sum_{r\in\bfZ}[f(w+\frac{r}{T})]p^{\frac{r}{T}} \in\breve{\bfQ}_{p,T}$. By the construction of $\wtilde{W}$, we have $C_w\neq 0$ and $f(w+\frac{r}{T})=0$ for every $w\in \wtilde{W}$ and every integer $r<0$. Thus $C_w\in \breve{\bfZ}_{p,T}$, and the element
    \begin{equation}\label{eq:20385}
        \widehat{f}\coloneqq \sum_{w\in\wtilde{W}}C_w\cdot t^w\in \breve{\bfZ}_{p,T}\pparen{t^\bfQ}
    \end{equation}
    is a lift of $\sigma(f)$ in $\breve{\bfZ}_{p,T}\pparen{t^\bfQ}$.

\section{Main theorem}\label{sec:main}

In this section, we prove the main theorem of this paper (cf. \Cref{thm:23791}): a $p$-adic Hahn series whose sign-inverted support, after scaling by some integer $T\geq 1$, admits a sparse set of representatives modulo $\bfZ$ is transcendental over $\breve{\bfQ}_p$. The argument proceeds by contradiction: if such a Hahn series $f$ is a root of a polynomial $P$ over $\breve{\bfQ}_p$, then expanding the lifted element $P(\widehat{f})$ from \Cref{eq:20385} via the multinomial formula would yield a family of vanishing identities on its coefficients (cf. \Cref{lem:Tnull-from-polynomial}). The sparseness condition then forces these identities, when specialized at a carefully chosen exponent, to collapse to a single nonzero term (cf. \Cref{lem:unique-phi-at-q0}), leading to a contradiction.

We isolate the two key inputs for proving \Cref{thm:23791} as separate lemmas as follows.
\begin{lemma}\label{lem:Tnull-from-polynomial}
    Let $f\in\bfL_p$ be a $p$-adic Hahn series and let $T\geq 1$ be an integer. Let $W\subset [0,1)\cap\bfQ$ be a set of representatives of $-T\cdot\opn{Supp}(f)$ modulo $\bfZ$, and let $\wtilde{W}\subseteq\opn{Supp}(f)$ and $\widehat{f}=\sum_{w\in\wtilde{W}}C_w\cdot t^w\in\breve{\bfZ}_{p,T}\pparen{t^\bfQ}$ be as in \Cref{eq:20385}. Suppose that $f$ is a root of a polynomial $P(X)=\sum_{i=0}^n a_iX^i\in\breve{\bfZ}_p[X]$. Then for every $q\in\bfQ$, one has
    \begin{equation}\label{eq:7609}
        \sum_{w\in \bfZ}p^{\frac{w}{T}}\sideset{}{'}\sumbb_{\substack{\wtilde{\phi}\colon\wtilde{W}\to \bfN\\\sumtt^{''}_{s\in\wtilde{W}}\wtilde{\phi}(s)\leq n\\\sumtt^{''}_{s\in\wtilde{W}}\wtilde{\phi}(s)\cdot s=q+\frac{w}{T}}}a_{\sumtt^{''}_{s\in\wtilde{W}}\wtilde{\phi}(s)}\frac{\left(\sumtt^{''}_{s\in\wtilde{W}}\wtilde{\phi}(s)\right)!}{\prodtt^{''}_{s\in\wtilde{W}}\wtilde{\phi}(s)!}\sideset{}{^{''}}\prodtt_{s\in\wtilde{W}}C_s^{\wtilde{\phi}(s)}=0,
    \end{equation}
    where $\sum^{'}$ emphasizes that the summation is taken over a finite index set, and $\sum^{''}$ (resp. $\prod^{''}$) emphasizes that the function inside the summation (resp. product) is finitely supported.
\end{lemma}
\begin{proof}
    By multinomial expansion, for any integer $i\geq 0$, one has
    $$\widehat{f}^i=\sum_{q\in\bfQ}\left(\sideset{}{'}\sumbb_{\substack{s_1,\cdots,s_i\in\wtilde{W}\\s_1+\cdots+s_i=q}}\prod_{k=1}^i C_{s_k}\right)t^q=\sum_{q\in\bfQ}\left(\sideset{}{'}\sumbb_{\substack{\wtilde{\phi}\colon\wtilde{W}\to \bfN\\\sumtt^{''}_{s\in\wtilde{W}}\wtilde{\phi}(s)=i\\\sumtt^{''}_{s\in\wtilde{W}}\wtilde{\phi}(s)\cdot s=q}}\frac{i!}{\prodtt^{''}_{s\in\wtilde{W}}\wtilde{\phi}(s)!}\sideset{}{^{''}}\prodtt_{s\in\wtilde{W}}C_s^{\wtilde{\phi}(s)}\right)t^q.$$
    Since $\sigma\left(\sum_{i=0}^n a_i f^i\right)=\sum_{i=0}^n a_i \sigma(f)^i=0$ and $\widehat{f}$ is a lift of $\sigma(f)$, the element
    \begin{align*}
        \sum_{i=0}^n a_i \widehat{f}^i= & \sum_{q\in\bfQ}\left(\sum_{i=0}^n a_i\sideset{}{'}\sumbb_{\substack{\wtilde{\phi}\colon\wtilde{W}\to \bfN \\\sumtt^{''}_{s\in\wtilde{W}}\wtilde{\phi}(s)=i\\\sumtt^{''}_{s\in\wtilde{W}}\wtilde{\phi}(s)\cdot s=q}}\frac{i!}{\prodtt^{''}_{s\in\wtilde{W}}\wtilde{\phi}(s)!}\sideset{}{^{''}}\prodtt_{s\in\wtilde{W}}C_s^{\wtilde{\phi}(s)}\right)t^q\\
        =                               & \sum_{q\in\bfQ}\left(\sideset{}{'}\sumbb_{\substack{\wtilde{\phi}\colon\wtilde{W}\to \bfN                 \\\sumtt^{''}_{s\in\wtilde{W}}\wtilde{\phi}(s)\leq n\\\sumtt^{''}_{s\in\wtilde{W}}\wtilde{\phi}(s)\cdot s=q}}a_{\sumtt^{''}_{s\in\wtilde{W}}\wtilde{\phi}(s)}\frac{\left(\sumtt^{''}_{s\in\wtilde{W}}\wtilde{\phi}(s)\right)!}{\prodtt^{''}_{s\in\wtilde{W}}\wtilde{\phi}(s)!}\sideset{}{^{''}}\prodtt_{s\in\wtilde{W}}C_s^{\wtilde{\phi}(s)}\right)t^q
    \end{align*}
    is a $T$-null-series in $\breve{\bfZ}_{p,T}\pparen{t^\bfQ}$. The equation \Cref{eq:7609} follows from the definition of $T$-null-series.
\end{proof}

\begin{lemma}\label{lem:unique-phi-at-q0}
    Let $f\in\bfL_p$ be a $p$-adic Hahn series. Let $S\subset \bbP$ be a $(c,n)$-sparse subset for some integers $c, n\geq 1$, and let $\uld^{(1)},\cdots,\uld^{(n)}\in S$ satisfy condition (2) of \Cref{def:38144}. Set $\phi_0\colon S\lto\bfN$, $\uld\longmapsto \opn{card}\left(\left\{i\,\middle\vert\,\uld=\uld^{(i)}\right\}\right)$. Suppose that $\norm{S}$ is a set of representatives of $-T\cdot\opn{Supp}(f)$ modulo $\bfZ$, let $$\wtilde{S}\coloneqq\wtilde{\norm{S}}=\left\{-\frac{\normd}{T}-n_q\middle\vert  \uld\in S\right\}$$ and $\mu\coloneqq\mu_{\norm{S}}\circ\norm{\cdot}\colon S\lto\wtilde{S}$ be as in \Cref{eq:mu-bijection}. Set
    $$q_0\coloneqq-\frac{1}{T}\sum_{\uld\in S}\normd\cdot\phi_0(\uld).$$
    Then $\phi_0\circ\mu^{-1}\colon\wtilde{S}\to\bfN$ is the unique function $\wtilde{\phi}$ satisfying
    $$\sideset{}{^{''}}\sumtt_{s\in\wtilde{S}}\wtilde{\phi}(s)\leq n\quad\text{and}\quad \sideset{}{^{''}}\sumtt_{s\in\wtilde{S}}\wtilde{\phi}(s)\cdot s\equiv q_0\pmod{\tfrac{1}{T}\bfZ}.$$
\end{lemma}
\begin{proof}
    Take $\wtilde{\phi}\colon\wtilde{S}\to \bfN$ that satisfies $\sumtt^{''}_{s\in\wtilde{S}}\wtilde{\phi}(s)\leq n$ and $\sumtt^{''}_{s\in\wtilde{S}}\wtilde{\phi}(s)\cdot s\equiv q_0\pmod{\frac{1}{T}\bfZ}$, and set $\phi\coloneqq \wtilde{\phi}\circ \mu$.
    Then
    $$\sideset{}{^{''}}\sumtt_{\underline{d}\in S}\phi(\uld)=\sideset{}{^{''}}\sumtt_{s\in\wtilde{S}}\wtilde{\phi}(s)\leq n$$
    and
    \begin{align*}
        \sideset{}{^{''}}\sumtt_{\uld\in S}\phi(\uld)\cdot \normd= & -T\sideset{}{^{''}}\sumtt_{s\in\wtilde{S}}\wtilde{\phi}(s)\cdot s -\sideset{}{^{''}}\sumtt_{\uld\in S}\wtilde{\phi}\left(\mu(\uld)\right)\left(T\cdot n_{\normd}\right) \\
        \equiv                                                   & -T\sideset{}{^{''}}\sumtt_{s\in\wtilde{S}}\wtilde{\phi}(s)\cdot s\\
        \equiv&-T\cdot q_0= \sum_{\uld\in S}\normd\cdot\phi_0(\uld)\pmod{\bfZ}.
    \end{align*}
    By \Cref{lem:39609}, one knows that $\phi=\phi_0$, i.e. $\wtilde{\phi}=\phi_0\circ \mu^{-1}$.
\end{proof}
The main theorem now follows from the above two lemmas and the construction of $\widehat{f}$ in \Cref{eq:20385}.
\begin{theorem}\label{thm:23791}
    Let $f\in\bfL_p$ be a $p$-adic Hahn series such that there exists an integer $T\geq 1$ for which $-T\cdot\opn{Supp}(f)$ admits a sparse set $W\neq \{0\}$ of representatives modulo $\bfZ$. Then $f$ is transcendental over $\breve{\bfQ}_p$, and hence over $\bfQ_p$.
\end{theorem}
\begin{proof}
    Suppose for contradiction that $f$ is a root of a polynomial $P(X)=a_nX^n+a_{n-1}X^{n-1}+\cdots+a_0\in \breve{\bfZ}_p[X]$ with $a_n\neq 0$. By multiplying $P(X)$ by a suitable power of $X$, we may assume that there exists $S\subset\bbP$ such that $\norm{S}\coloneqq\{\normd\in\bfQ\vert \uld\in S\}$ is a set of representatives of $-T\cdot\opn{Supp}(f)$ modulo $\bfZ$, and $S$ is $(c,n)$-sparse for some integer $c\geq 1$ (cf. \Cref{lem:54409}), where $n$ equals the degree of $P(X)$. Let $\uld^{(1)},\cdots,\uld^{(n)}\in S$ be elements satisfying condition (2) of \Cref{def:38144} for $n$ and $c$, set
    $$\phi_0\colon S\lto\bfN,\ \uld\longmapsto \opn{card}\left(\left\{i\middle\vert 1\leq i\leq n,\ \uld=\uld^{(i)}\right\}\right),$$
    and let $\wtilde{S}\coloneqq\wtilde{\norm{S}}$ and $\mu\colon S\to\wtilde{S}$ be as in \Cref{eq:mu-bijection} (with $W=\norm{S}$).

    By \Cref{lem:Tnull-from-polynomial} applied to $W=\norm{S}$, for every $q\in\bfQ$ one has the identity \Cref{eq:7609}. Specializing to
    $$q=q_0\coloneqq-\frac{1}{T}\sum_{\uld\in S}\normd\cdot\phi_0(\uld),$$
    \Cref{lem:unique-phi-at-q0} tells us that the only $\wtilde{\phi}\colon\wtilde{S}\to\bfN$ contributing to the interior summation of \Cref{eq:7609} at $q=q_0$ is $\wtilde{\phi}=\phi_0\circ\mu^{-1}$. Hence \Cref{eq:7609} at $q=q_0$ reduces to the single term
    $$\frac{p^{\frac{w}{T}}\cdot n!}{\prodtt^{''}_{\uld\in S}\phi_0(\uld)!}\cdot a_n\cdot\sideset{}{^{''}}\prodtt_{\uld\in S}C_{\mu(\uld)}^{\phi_0(\uld)}=0,$$
    where $w=T\left(\sumtt^{''}_{s\in\wtilde{S}}\phi_0\left(\mu^{-1}(s)\right)\cdot s-q_0\right)\in\bfZ$. This forces $a_n=0$ or $C_{\mu(\uld)}=0$ for some $\uld\in S$, contradicting $a_n\neq 0$ and $C_s\neq 0$ for every $s\in \wtilde{S}$.
\end{proof}
\section{Application: $p$-adic Hahn series with bounded support}\label{sec:55224}
The goal of this section is to prove \Cref{thm:16129} (restated as \Cref{coro:11594} below): a $p$-adic Hahn series that is algebraic over $\bfQ_p$ and whose support is bounded with only finitely many accumulation points must have finite support. The argument proceeds in three steps. First, building on results of Kedlaya, we show that $\bfQ_p$-algebraicity forces the coefficient function to be \emph{quasi-twist-recurrent} (QTR), a combinatorial recurrence condition (\Cref{prop:167}). Second, we analyze the structure of bounded QTR sets and prove that, when the number of accumulation points is finite, such a set decomposes into a finite set together with finitely many pairwise disjoint \emph{rays} (\Cref{coro:48108}). Finally, we show that this ray decomposition yields a sparse set of representatives for the sign-inverted and suitably scaled support (\Cref{lem:57121,lem:42556}), so that our main transcendence theorem (\Cref{thm:23791}) forces the series to be transcendental unless its support is finite.

In this section, we will write
$$0.q_1\cdots q_n\cdots\coloneqq \sum_{i=1}^\infty q_i\cdot p^{-i},\ q_i\in\{0,1,\cdots,p-1\}$$ 
to represent the base $p$ expansion of any rational number in $[0,1)$. We will require that $q_j\neq p-1$ for infinitely many $j$. In fact, every such expansion occurring in this section has finite length, that is, $q_i=0$ for $i\gg 1$. 

In addition, for any Hahn series $f=\sum_{q\in\bfQ}f(q)t^q\in\bfL_p^\flat$ (resp. $f=\sum_{q\in\bfQ}[f(q)]p^q\in\bfL_p$), we will write $F_f$ to denote the coefficient function $q\longmapsto f(q)$ from $\bfQ$ to $\overline{\bfF}_p$.

\subsection{Quasi-twist-recurrent functions}\label{sec:5330}
In \cite{kedlayaAlgebraicClosurePower2001}, Kedlaya characterizes the algebraic closure of $K\pparen{t}$ in $K\pparen{t^\bfQ}$ for an arbitrary algebraically closed field $K$ of characteristic $p>0$ in terms of twist-recurrent sequences. In particular, when $K=\overline{\bfF}_p$, this characterization admits a simpler description\footnote{See also \cite[Remark 2.7, Remark 2.9]{kedlayaAlgebraicityGeneralizedPower2017b} for the critical remarks on \cite[Theorem 15]{kedlayaAlgebraicClosurePower2001}.}:
\begin{theorem}[{cf. \cite[Theorem 15]{kedlayaAlgebraicClosurePower2001}\cite[Theorem 11.11]{kedlayaAlgebraicityGeneralizedPower2017b}}]\label{thm:28713}
    For $a\in\bfZ_{>0}$, $b\in\bfZ$ and $c\in\bfZ_{\geq 0}$, we define
    $$S_{a,b,c}=\left\{\frac{1}{a}\left(n-\sum_{k=1}^\infty b_kp^{-k}\right)\middle\vert n\geq -b,\ b_k\in \bfN_{<p},\ \sum b_k\leq c\right\}\subset \frac{1}{a}\bfZ[1/p]$$
    and
    $$T_c\coloneqq S_{1,0,c}\cap (-1,0)=\left\{-\sum_{k=1}^\infty b_kp^{-k}\middle\vert b_k\in \bfN_{<p},\ \sum b_k\leq c\right\}\subset\bfZ[1/p].$$
    A Hahn series $x=\sum_q x_q t^q\in\overline{\bfF}_p((t^\bfQ))$ is integral over $\overline{\bfF}_p((t))$ if and only if the following conditions hold:
    \begin{enumerate}
        \item\label{it:23822} There exist $a,b,c$ such that $\opn{Supp}(x)\subseteq S_{a,b,c}$.
        \item For some (hence any) $a,b,c$ as in \Cref{it:23822}, there exist integers $M,N$ with the following property: for each integer $m\geq -b$, the function $f_m\colon T_c\lto \overline{\bfF}_p$ given by $f_m(z)=x_{(m+z)/a}$ has the property that every sequence of the following form becomes periodic of period $N$ after at most $M$ terms:
              $$c_n=f_m\left(-\sum_{k=1}^{j-1}b_k p^{-k}-p^{-n}\left(\sum_{k=j}^\infty b_kp^{-k}\right)\right),\ n=0,1,\cdots,$$
              where $j\in\bfZ_{\geq 1}$ and $b_k\in\bfN_{<p}$ with $\sum b_k\leq c$ are arbitrary.
    \end{enumerate}
\end{theorem}

We rephrase this result by introducing the notion of quasi-twist-recurrent functions:
\begin{definition}\label{def:16557}
    We say a function $x\colon\bfQ\lto\overline{\bfF}_p$ is \textbf{quasi-twist-recurrent} (QTR) with respect to the data $(a,b,c,M,N)\in \bfZ_{>0}\times \bfN\times \bfN\times \bfZ_{>0}\times \bfZ_{>0}$, if its support $\opn{Supp}(x)$ is a well-ordered subset of $\bfQ$ such that %
    \begin{enumerate}
        \item for any $q\in\opn{Supp}(x)$, if we expand $aq$ in base $p$ as
              $$aq=w-0.q_1\cdots q_n\cdots,\ w\in\bfZ,\ q_n\in\{0,\cdots,p-1\}$$
              with $\sum q_n<\infty$, then $w\geq -b$ and $\sum q_n\leq c$;
        \item for any integer $w\geq -b$ and any rational number $q$ of the form
              $$q=\frac{1}{a}\left(w-0.q_1\cdots q_n\cdots\right),\ q_n\in\bfN_{<p},\ \sum q_n\leq c,$$
              if there exist $M$ consecutive $0$ in the sequence $q_1,q_2,\cdots$, i.e. $q_{k+1}=q_{k+2}=\cdots=q_{k+M}=0$ for some integer $k\geq 0$, then the coefficient of $x$ at $q$ equals the coefficient of $x$ at the following rational number:
              $$\frac{1}{a}\left(w-0.q_1\cdots q_k\overbrace{0\cdots0}^{M+N}q_{k+M+1}\cdots\right).$$
    \end{enumerate}
    We omit the data $(a,b,c,M,N)$ when it is clear from the context, and we say a subset of $\bfQ$ is QTR if it is the support of a QTR function.
\end{definition}
\Cref{thm:28713} can be restated as follows:
\begin{proposition}\label{prop:54845}
    A Hahn series $x=\sum_q x_q t^q\in\overline{\bfF}_p((t^\bfQ))$ is algebraic over $\overline{\bfF}_p((t))$ if and only if the function $F_x$ is QTR.
\end{proposition}

On the other hand, Kedlaya proved the following result, which links the algebraicity of $p$-adic Hahn series to that of equal-characteristic Hahn series:
\begin{theorem}[{cf. \cite[Theorem 13.4]{kedlayaAlgebraicityGeneralizedPower2017b}}]\label{thm:47}
    Let $L$ be the completed integral closure of $\overline{\bfF}_p((t))$ in $\overline{\bfF}_p((t^\bfQ))$. Then the completed integral closure of $\breve{\bfQ}_p$ in $\bfL_p$ is the completion of the following set:
    $$\left\{\sum_q[x_q]p^q\in\bfL_p\middle\vert\sum_q x_q t^q\in L\right\}.$$
\end{theorem}
\begin{remark}
    Since the field $\bfC_p$ is complete and algebraically closed, it identifies with the $p$-adic completion of the algebraic (integral) closure of $\bfQ_p$. Conseqently, \Cref{thm:47} provides a complete description of the expansion of $p$-adic complex numbers in $\bfL_p$.
\end{remark}
We observe that the combination of \Cref{prop:54845} and \Cref{thm:47} yields a necessary condition for a $p$-adic Hahn series with bounded support to be algebraic over $\breve{\bfQ}_p$:
\begin{proposition}\label{prop:167}
    Let $f=\sum_{q\in\bfQ}[f(q)]p^q\in\bfL_p$ be a $p$-adic Hahn series with bounded support. If $f$ is algebraic over $\breve{\bfQ}_p$, then the function $F_f$ is QTR.
\end{proposition}

This is a direct consequence of the following lemma:
\begin{lemma}\label{lem:36014}
    For any QTR function $\phi\colon\bfQ\lto\overline{\bfF}_p$, and any integer $r$, the restriction of $\phi$ to $(-\infty,r]\cap\bfQ$ is still QTR.
\end{lemma}
\begin{proof}
    Write $\phi_r$ for the restriction of $\phi$ to $(-\infty,r]$, so that $\phi_r(q)=\phi(q)$ for $q\leq r$ and $\phi_r(q)=0$ for $q>r$; in particular $\phi_r$ and $\phi$ agree on $(-\infty,r]$ and $\opn{Supp}(\phi_r)=\opn{Supp}(\phi)\cap(-\infty,r]$. We check that $\phi_r$ is QTR with the same data $(a,b,c,M,N)$ as $\phi$.

    Since $\opn{Supp}(\phi_r)\subseteq\opn{Supp}(\phi)$ is well-ordered and the first condition of \Cref{def:16557} only constrains elements of the support, it holds for $\phi_r$ with the same $a,b,c$.

    For the second condition, fix an integer $w\geq -b$ and a rational number
    $$q=\frac{1}{a}\left(w-0.q_1\cdots q_n\cdots\right),\qquad q_n\in\bfN_{<p},\ \sum q_n\leq c,$$
    with $q_{k+1}=\cdots=q_{k+M}=0$ for some integer $k\geq 0$, and set
    $$q'\coloneqq\frac{1}{a}\left(w-0.q_1\cdots q_k\overbrace{0\cdots0}^{M+N}q_{k+M+1}\cdots\right).$$
    We must show $\phi_r(q)=\phi_r(q')$. Inserting $N$ extra zeros shifts the digits $q_{k+M+1},q_{k+M+2},\dots$ to the right, so the decimal part of $aq'$ is at most that of $aq$; hence $aq'\geq aq$, i.e.\ $q'\geq q$.

    If $q>r$, then $q'\geq q>r$, so $\phi_r(q)=0=\phi_r(q')$.

    If $q\leq r$, then from $aq=w-0.q_1\cdots\leq ar$ together with $0\leq 0.q_1\cdots<1$ and $w,ar\in\bfZ$ we get $w\leq ar$; therefore $aq'=w-0.q_1\cdots q_k\overbrace{0\cdots0}^{M+N}q_{k+M+1}\cdots\leq w\leq ar$, so $q'\leq r$ as well. Then $\phi_r(q)=\phi(q)$ and $\phi_r(q')=\phi(q')$, which are equal because $\phi$ satisfies the second condition of \Cref{def:16557}.

    In either case $\phi_r(q)=\phi_r(q')$, so $\phi_r$ is QTR.
\end{proof}

\begin{proof}[Proof of \Cref{prop:167}]
    Suppose that $\opn{Supp}(f)$ is bounded from above by some $u\in \bfZ$. Then \Cref{thm:47} shows that there exists $f'\in\bfL_p$ such that $F_f$ and $F_{f'}$ coincide on $(-\infty,u]$ and $f'$ lies in the completed integral closure of $\overline{\bfF}_p((t))$. Similarly, there exists $f''\in \bfL_p^\flat$ such that $F_{f'}$ and $F_{f''}$ coincide on $(-\infty,u]$ and $f''\in \overline{\bfF}_p((t))^{\opn{alg}}$. By \Cref{prop:54845}, the function $F_{f''}$ is QTR, and the result follows from \Cref{lem:36014}.
\end{proof}

\subsection{Ray decomposition of the QTR sets}\label{sec:23337}
Although the definition of QTR sets may appear complicated, we show in this section that if one imposes restrictions on their order type, then a bounded QTR set can be decomposed into finitely many pieces with a very simple structure, which we call the ray decomposition.

Throughout \Cref{sec:23337}, we fix a set $S\subseteq S_{a,b,c}$ with the following properties:
\begin{enumerate}[label=(S\arabic*),ref=\textup{S\arabic*}]
    \item\label{it:65165} %
          $S\subseteq S_{a,b,c,m}$ for some integer $m\geq -b$, where
$$S_{a,b,c,m}=\left\{\frac{1}{a}\left(m-\sum_{k=1}^\infty q_kp^{-k}\right)\middle\vert q_k\in \bfN_{<p},\ \sum q_k\leq c\right\}\subset S_{a,b,c}.$$
    \item\label{it:49436} There exists $M,N\in \bfN$ such that if an element $q\in S$ of the form
          $$q=\frac{1}{a}\left(m-0.q_1\cdots q_n\cdots\right),\ q_n\in\bfN_{<p},\ \sum q_n\leq c$$
          satisfies that there exist $M$ consecutive $0$ in the sequence $q_1,q_2,\cdots$, i.e. $q_{k+1}=q_{k+2}=\cdots=q_{k+M}=0$ for some integer $k\geq 0$, then the following rational number also belongs to $S$:
          $$\frac{1}{a}\left(m-0.q_1\cdots q_k\overbrace{0\cdots0}^{M+N}q_{k+M+1}\cdots\right).$$
    \item\label{it:31201} The order type of $S$ is strictly less than $\omega^2$. In other words, $S$ admits only finitely many accumulation points.
\end{enumerate}
Such a set $S$ is called $(a,b,c,m,M,N)$-admissible.

\begin{remark}
    Note that the condition \Cref{it:49436} is strictly weaker than the one we use to define QTR sets. On the other hand, as shown in the course of the proof of \Cref{lem:42556}, bounded QTR sets are always a finite disjoint union of QTR sets that are contained in some $S_{a,b,c,m}$ for some integer $m\geq -b$. Thus the assumption $S\subseteq S_{a,b,c,m}$ entails no loss of generality when we study the structure of bounded QTR sets.
\end{remark}

We begin by defining the concepts of words and gap vectors:
\begin{definition}\leavevmode
    \begin{enumerate}
        \item By a \textbf{word}, we mean an ordered tuple of elements in $\{1,\cdots,p-1\}$ of finite length. For an element $q\in S$ of the form
              $$q=\frac{1}{a}\left(m-0.q_1\cdots q_n\cdots\right),$$
              the \textbf{word} of $q$ is the tuple of nonzero digits of $0.q_1\cdots q_n\cdots$.
        \item By a \textbf{gap vector}, we mean an ordered tuple of natural numbers of finite length. For an element
              $$q=\frac{1}{a}\left(m-0.q_1\cdots q_n\cdots\right)$$
              of $S$, the $i$-th coordinate of the \textbf{gap vector} of $q$ is the number of consecutive $0$ between the $(i-1)$-th and $i$-th nonzero digits of $0.q_1\cdots q_n\cdots$. In particular, the first coordinate of the gap vector of $q$ is the number of consecutive $0$ before the first nonzero digit of $0.q_1\cdots q_n\cdots$.
    \end{enumerate}
\end{definition}
\begin{example}
    Let $q=\frac{1}{a}(m-0.001104514)$, then the word of $q$ is $(1,1,4,5,1,4)$ and the gap vector of $q$ is $(2,0,1,0,0,0)$.
\end{example}

Since the value of $0.q_1\cdots q_n\cdots$ is determined by the positions and the values of its nonzero digits, an element of $S$ is uniquely determined by the pair consisting of its word and its gap vector. We shall freely identify an element of $S$ with this pair; note that the zeros following the last nonzero digit of $0.q_1\cdots q_n\cdots$ are not recorded by the gap vector.

The following lemma is a direct consequence of \Cref{it:65165}:
\begin{lemma}\label{lem:5966}
    For any element of $S$, the length of its word is at most $c$. In particular, there are only finitely many possible words for elements of $S$.
\end{lemma}
For any integer $t\geq 1$, denote by $\bfe^{(i)}$ the $i$-th standard basis vector of $\bfN^t$. The condition \Cref{it:49436} can be rephrased as follows:
\begin{lemma}\label{lem:30089}
    Fix a word $\mathbf{d}=(d_1,\ldots,d_t)$ and let $G_{\mathbf{d}}\subseteq \bfN^t$ be the set of gap vectors of elements of $S$ with word $\mathbf{d}$. If $g\in G_{\bfd}$ and $g_i\geq M$ for some $i\leq t$, then $g+N \bfe^{(i)}\in G_{\bfd}$.
\end{lemma}
We observe that the ordinal bound of $S$ (i.e. condition \Cref{it:31201}) imposes a strong restriction on the gap vectors of elements of $S$:
\begin{lemma}\label{lem:27279}
    For every word $\bfd$ occurring in $S$ and any gap vector $g\in G_{\bfd}$, at most one coordinate of $g$ is at least $M$.
\end{lemma}
\begin{proof}
    Suppose there exist $i<j$ such that $g_i\geq M$ and $g_j\geq M$. Then by \Cref{lem:30089}, one has $g+s\bfe^{(i)}+tN\bfe^{(j)}\in G_{\bfd}$ for any $s,t\in\bfN$. We set $v_{s,t}$ to be the element of $S$ with word $\bfd$ and gap vector $g+s\bfe^{(i)}+tN\bfe^{(j)}$. For any fixed $s$, the sequence $\{v_{s,t}\}_{t\in\bfN}$ is strictly increasing and consequently has order type $\omega$. On the other hand, for any $s<s'$, one has $v_{s,t}<v_{s',t'}$ for any $t,t'\in\bfN$. Therefore, the sequence $\{v_{s,t}\}_{s,t\in\bfN}$ has order type $\omega^2$, which contradicts the assumption on $S$.
\end{proof}
\begin{lemma}\label{lem:19048}
    For any fixed word $\bfd$, one can write
    $$G_{\bfd}=W\cup A_1\cup\cdots \cup A_r,$$
    where $W$ is a finite set, and each $A_i$ has the form $\left\{v^{(i)}+kN \bfe^{(u)}\middle\vert k=0,1,2,\cdots\right\}$, where $v^{(i)}\in G_{\bfd}$ satisfies $v_j^{(i)}\geq M$ if and only if $j=u$.
\end{lemma}
\begin{proof}
    Take $W$ to be the set of gap vectors in $G_{\bfd}$ with all coordinates less than $M$. Since the gap vectors have fixed length $t$ and each coordinate then lies in $\{0,1,\ldots,M-1\}$, the set $W$ is finite.

    By \Cref{lem:27279}, every $g\in G_{\bfd}\backslash W$ has \emph{exactly} one coordinate that is at least $M$; denote its index by $u(g)\in\{1,\ldots,t\}$. To such a $g$ we attach the datum
    $$\delta(g)=\left(u(g),\ (g_k)_{k\neq u(g)},\ g_{u(g)}\bmod N\right).$$
    Since $u(g)\in\{1,\ldots,t\}$, each $g_k$ with $k\neq u(g)$ lies in $\{0,\ldots,M-1\}$, and the last entry lies in $\{0,\ldots,N-1\}$, the datum $\delta(g)$ ranges over a finite set. Let $\delta_1,\ldots,\delta_r$ be the values actually attained and put $A_j=\left\{g\in G_{\bfd}\backslash W\mid \delta(g)=\delta_j\right\}$, so that
    $$G_{\bfd}=W\cup A_1\cup\cdots\cup A_r.$$

    Fix $j$ and let $u$ be the index recorded in $\delta_j$. Any two elements of $A_j$ agree in every coordinate other than the $u$-th and have $u$-th coordinates congruent modulo $N$, hence differ by an integer multiple of $N\bfe^{(u)}$. Let $v^{(j)}$ be the element of $A_j$ with smallest $u$-th coordinate. Then
    $$A_j\subseteq\left\{v^{(j)}+kN\bfe^{(u)}\mid k=0,1,2,\ldots\right\}.$$
    Conversely $v^{(j)}_u\geq M$, so repeated application of \Cref{lem:30089} gives $v^{(j)}+kN\bfe^{(u)}\in G_{\bfd}$ for every $k\geq0$; each such vector has datum $\delta_j$ and therefore lies in $A_j$. Thus $A_j=\left\{v^{(j)}+kN\bfe^{(u)}\mid k\geq0\right\}$ has the required form, with $v^{(j)}_l\geq M$ if and only if $l=u$.
\end{proof}
\begin{definition}\label{def:33761}
    A \textbf{ray} in $S$ is a subset of $S$ of the form
    \begin{equation}\label{eq:16714}
        \left\{\frac{1}{a}\left(m-0.q_1q_2\cdots q_{s}\overbrace{0\cdots 0}^{kN}q_{s+1}\cdots\right)\middle\vert k=0,1,\cdots\right\},
    \end{equation}
    where $\frac{1}{a}(m-0.q_1q_2\cdots q_n\cdots)$ is a fixed element of $S$ such that $0.q_{s+1}q_{s+2}\cdots\neq 0$ and there is one and only one coordinate of the gap vector of it that is at least $M$.
\end{definition}
\begin{remark}\label{rem:ray-progression}
    Under the identification of an element of $S$ with its (word, gap vector) pair, inserting $kN$ zeros at the unique gap of length at least $M$, say the $u$-th one, corresponds to adding $kN\bfe^{(u)}$ to the gap vector. Hence, for a fixed word $\bfd$, the rays in $S$ with word $\bfd$ are precisely the arithmetic progressions $\{v+kN\bfe^{(u)}\mid k\geq 0\}$ furnished by \Cref{lem:19048}. Geometrically, a ray as \eqref{eq:16714} is a strictly increasing sequence converging to its limit point $\frac{1}{a}(m-0.q_1\cdots q_{s})\in\bfQ$; this limit point is an accumulation point of $S$, which explains why the finiteness of the accumulation points (condition \Cref{it:31201}) controls the number of rays.
\end{remark}
\begin{corollary}\label{coro:8081}
    The set $S$ can be written as a union of a finite set and finitely many rays.
\end{corollary}
\begin{proof}
    Different decimals $0.q_1q_2\cdots q_n\cdots$ in \Cref{def:33761} give rise to disjoint rays. The result follows from \Cref{lem:5966} and \Cref{lem:19048}.
\end{proof}
In fact, the decomposition in \Cref{coro:8081} can be further refined to a disjoint union of a finite set and finitely many rays. This is guaranteed by the following lemma:
\begin{lemma}\label{lem:59667}
    Let $R_1, R_2$ be two rays in $S$. Then exactly one of the following cases can happen:
    \begin{enumerate}
        \item $R_1\subseteq R_2$ or $R_2\subseteq R_1$;
        \item $R_1\cap R_2$ is a finite set.
    \end{enumerate}
\end{lemma}
\begin{proof}
    Write
    $$R_1=\left\{\frac{1}{a}\left(m-\alpha_1-\lambda_1 p^{-Nk}\right)\middle\vert k=0,1,\cdots\right\},\ R_2=\left\{\frac{1}{a}\left(m-\alpha_2-\lambda_2 p^{-Nk}\right)\middle\vert k=0,1,\cdots\right\},$$
    where $\alpha_1,\alpha_2\in [0,1)$ and $\lambda_1,\lambda_2\in (0,1)$ are rational numbers with finite length base-$p$ expansions. If $\alpha_1=\alpha_2$, then the condition $q\in R_1\cap R_2$ yields $$m-\alpha_1-aq=\lambda_1 p^{-Nk_1}=\lambda_2 p^{-Nk_2}$$
    for some $k_1,k_2\in\bfN$. Without loss of generality, we assume $k_1\leq k_2$. Then one has $\lambda_1 = \lambda_2 p^{-N(k_2-k_1)}$ and consequently $R_1\subseteq R_2$.

    If $\alpha_1\neq \alpha_2$ and $R_1\cap R_2$ is an infinite set, then $R_1\cap R_2$, as an infinite subsequence of $R_1$ (resp. $R_2$), converges to its accumulation point $\frac{1}{a}(m-\alpha_1)$ (resp. $\frac{1}{a}(m-\alpha_2)$), which forces $\alpha_1=\alpha_2$, a contradiction.
\end{proof}
\begin{proposition}\label{prop:29055}
    The set $S$ can be written as a union of a finite set and finitely many rays that are pairwise disjoint.
\end{proposition}
\begin{proof}
    By \Cref{coro:8081}, we may write $S=E_0\cup R_1\cup\cdots\cup R_s$, where $E_0$ is a finite set and $R_1,\cdots,R_s$ are rays. By \Cref{lem:59667}, after discarding every ray that is contained in another one, we may assume that no $R_l$ is contained in another; then $R_i\cap R_j$ is a finite set for all $i\neq j$. Each ray is a strictly increasing sequence parametrized by $k=0,1,\cdots$ as in \Cref{def:33761}; for $k_0\geq 0$, let $R_l'$ denote the tail of $R_l$ consisting of the elements with $k\geq k_0$, which is again a ray. Since there are finitely many pairs $(i,j)$ and each $R_i\cap R_j$ is finite, we may choose $k_0$ large enough that $R_1',\cdots,R_s'$ are pairwise disjoint. Then $\bigcup_{l=1}^s(R_l\setminus R_l')$ is a finite set, and setting $E\coloneqq E_0\cup\bigcup_{l=1}^s(R_l\setminus R_l')$ yields the pairwise disjoint decomposition $S=E\sqcup R_1'\sqcup\cdots\sqcup R_s'$.
\end{proof}
\begin{corollary}\label{coro:48108}
    Let $U$ be a bounded QTR set with respect to the data $(a,b,c,M,N)$. If $U$ has only finitely many accumulation points, then $U$ can be written as a union of a finite set and finitely many pairwise disjoint rays, each contained in some $(a,b,c,m,M,N)$-admissible subset of $U$ with $m\geq -b$.
\end{corollary}
\begin{proof}
    Since $U$ is bounded, there exist finitely many integers $m_i\geq -b$ such that $U\subseteq \bigcup_i S_{a,b,c,m_i}$. The result follows by applying \Cref{prop:29055} to each $S_{a,b,c,m_i}\cap U$, which is $(a,b,c,m_i,M,N)$-admissible and has finitely many accumulation points, together with the fact that the sets $S_{a,b,c,m_i}\cap U$ are disjoint for different $m_i$.
\end{proof}

\subsection{Sparse representatives and finiteness of bounded QTR supports}\label{sec:24235}
\begin{lemma}\label{lem:57121}
    Let $N\geq 1$ be an integer, and let $\delta_1,\cdots,\delta_r\in (0,1)$ with $r\geq 1$ be rational numbers with finite length base-$p$ expansions. Then the set
    $$W=\left\{\delta_ip^{-Nk}\middle\vert1\leq i\leq r,\ k=0,1,2\cdots\right\}$$
    is sparse.
\end{lemma}
\begin{proof}
    Take a large enough integer $s\geq 1$ such that $\delta_i=0.\delta_{i,1}\delta_{i,2}\cdots\delta_{i,s}$ for every $i$. Let
    $$C\coloneqq\max\left\{\sum_{j=1}^s\delta_{i,j}\middle\vert i=1,\cdots,r\right\}$$
    be the maximal digit sum of the $\delta_i$. Since the $\delta_i$ are nonzero, one has $C\geq 1$. Without loss of generality, we assume the digit sum of $\delta_1$ is $C$.

    Fix an integer $n\geq 1$. Set $d_i=\delta_1\cdot p^{-(2s+1)Ni}\in W$ for $i=1,\cdots,n$. Then $d_i\in \opn{Dom}_p(W)$. On the other hand, the exponent $p^{-(2s+1)Ni}$ spreads the digits cluster of different $d_i$ so apart that for any positive integers $i_1<i_2$, one has
    $$\opn{dist}\left(\{j\in\bfZ_{\geq}|\uld^{(i_1)}_j\neq 0\},\{j\in\bfZ_{\geq 1}|\uld^{(i_2)}\neq 0\}\right)>s,$$
    where $\opn{dist}(A,B)\coloneqq \inf_{a\in A, b\in B}\lvert a-b\rvert$ is the distance between two subsets $A, B$ of $\bfQ$. In particular, there is no carry in base $p$ when adding $d_1,\cdots,d_n$ together.

    Suppose $e_1,\cdots,e_n$ are elements of $W$ such that $d_1+\cdots+d_n-(e_1+\cdots+e_n)$ is an integer. Let $\uld^{(i)}$ (resp. $\underline{e}^{(i)}$) be the element in $\bbP$ such that $\norm{\uld^{(i)}}=d_i$ and $\norm{\underline{e}^{(i)}}=e_i$. Then
    \begin{equation}\label{eq:14517}
        \frakN_p\left(\sum_i e_i\right)=\Psi\left(\tau\left(\sum_i \underline{e}^{(i)}\right)\right)\leq \Psi\left(\sum_i \underline{e}^{(i)}\right)=\sum_i \Psi\left(\underline{e}^{(i)}\right)\leq n\cdot C=\frakN_p\left(\sum_i d_i\right).
    \end{equation}
    Since $d_1+\cdots+d_n-(e_1+\cdots+e_n)$ is an integer, we know that $\frakN_p\left(\sum_i e_i\right)=\frakN_p\left(\sum_i d_i\right)$ and consequently all inequalities in \Cref{eq:14517} are equalities. In particular, one has $\frakN_p(e_i)=\Psi(\underline{e}^{(i)})=C$ for each $i$, and there is no carry in base $p$ when adding $e_1,\cdots,e_n$ together. As a result, we know that $\sum_i d_i=\sum_i e_i$.

    Fix $l\in\bfZ_{\geq 1}$. Since $\opn{diam}(\{j\in\bfZ_{\geq 1}|\underline{e}^{(l)}_j\neq 0\})$ does not exceed $s$, we conclude that the set
    $$\left\{i\middle\vert\exists w\in\bfZ_{\geq 1},\ \uld^{(i)}_w\neq 0\wedge \underline{e}^{(l)}_w\neq 0\right\}$$
    has at most one element, i.e. the digits of $e_l$ overlap with those of $d_i$ for at most one $i$.

    On the other hand, the conditions
    \begin{enumerate}[label=(\arabic*),ref=\textup{(\arabic*)}]
        \item there is no carry in base $p$ when adding $d_1,\cdots,d_n$ (resp. $e_1,\cdots,e_n$) together;
        \item $d_1,\cdots,d_n$ (resp. $e_1,\cdots,e_n$) have the same digit sum $C$;
        \item $\sum_i d_i=\sum_i e_i$;
    \end{enumerate}
    ensure that $e_l$ is identical to certain $d_i$. We conclude that $e_1,\cdots,e_n$ are just a permutation of $d_1,\cdots,d_n$. This shows that $W$ is sparse.
\end{proof}

\begin{lemma}\label{lem:42556}
    Let $f\in\bfL_p$ be a $\breve{\bfQ}_p$-algebraic $p$-adic Hahn series such that the support is bounded and admits finitely many (at least one) accumulation points. Then there exists $S'\subseteq \opn{Supp}(f)$, $\lambda\in\bfQ$ and an integer $T\geq 1$ such that $\opn{Supp}(f)\backslash S'$ is a finite set and $-T(S'-\lambda)$ admits a sparse nonzero set of representatives modulo $\bfZ$.
\end{lemma}
\begin{proof}
    \Cref{prop:167} implies that the coefficient function $F_f$ of $f$ is QTR for certain $(a,b,c,M,N)\in \bfZ_{>0}\times \bfN\times \bfN\times \bfZ_{>0}\times \bfZ_{>0}$. This implies that $\opn{Supp}(f)\subseteq S_{a,b,c}$. Since $\opn{Supp}(f)$ is bounded, there exists a finite index set $I\subseteq \bfZ\cap [-b,\infty)$ such that $\opn{Supp}(f)\subseteq\bigcup_{m\in I} S_{a,b,c,m}$. For every $m\in I$, we set $S_m=\opn{Supp}(f)\cap S_{a,b,c,m}$. Then $S_m$ is $(a,b,c,m,M,N)$-admissible. By \Cref{prop:29055}, $S_m$ can be written as the pairwise disjoint union of a finite set and finitely many rays in $S_m$. Since the rays are also disjoint across different $m$, we set $S''=\bigsqcup_{l=1}^s R_l$ to be the (nonempty) union of all rays in $S_m$ for all $m\in I$. Then $\opn{Supp}(f)\backslash S''$ is a finite set, and $S''$ is a disjoint union of finitely many rays. For each ray $R_l$, we write
    $$R_l=\left\{a^{-1}\left(m_l-\alpha_l-\beta_l p^{-Nk}\right)\middle\vert k\in\bfN\right\},$$
    where $m_l\geq -b$ is an integer and $\alpha_l\in [0,1), \beta_l\in (0,1)$ are rational numbers with finite base-$p$ expansions. Note that $R_l$ has rational limit point $\lambda_l\coloneqq a^{-1}(m_l-\alpha_l)$.

    Choose an integer $T\geq 1$ such that $a$ divides $T$ and $T(\lambda_l-\lambda_1)\in\bfZ$ for every $l=1,2,\cdots,s$. Then for any $l=1,\cdots,s$ and any $q=a^{-1}(m_l-\alpha_l-\beta_l p^{-Nk})\in R_l$,
    $$-T(q-\lambda_1)=-T(\lambda_l-\lambda_1)+\frac{T}{a}\beta_l p^{-Nk}\equiv \gamma_l p^{-Nk}\pmod{\bfZ},$$
    where $\gamma_l=\frac{T}{a}\beta_l$ is a nonzero rational number. Since $\frac{T}{a}$ is an integer and $\beta_l$ has finite base-$p$ expansion, $\gamma_l$ has finite base-$p$ expansion, and hence $\gamma_l\in\bfZ[1/p]$. Choose a large enough integer $K\geq 0$ such that $0<\gamma_l p^{-NK}<1$ for every $l=1,\cdots,s$. We set $\delta_l=\gamma_l p^{-NK}$ for every $l=1,\cdots,s$.

    We delete the first $K$ points of each ray in $S''$ and denote the resulting set by $S'$, i.e.
    $$S'=\left\{a^{-1}\left(m_l-\alpha_l-\beta_l p^{-N(j+K)}\right)\middle\vert l=1,\cdots,s,\ j\in\bfN\right\}.$$
    Since this is a finite deletion, $\opn{Supp}(f)\backslash S'$ is still a finite set. On the other hand, for any $q=a^{-1}(m_l-\alpha_l-\beta_l p^{-N(j+K)})\in S'$, one has
    $$-T(q-\lambda_1)\equiv \delta_l p^{-Nj} \pmod{\bfZ}.$$
    Conversely, for any $l=1,\cdots,s$ and any $j\in\bfN$, there exists $q=a^{-1}(m_l-\alpha_l-\beta_l p^{-N(j+K)})\in S'$ such that
    $$-T(q-\lambda_1)\equiv \delta_l p^{-Nj}\pmod{\bfZ}.$$
    Thus the elements $\delta_l p^{-Nj}$ for $l=1,\cdots,s$ and $j=1,2,\cdots$, after deleting the duplicates, form a nonzero set of representatives of $-T(S'-\lambda_1)$ modulo $\bfZ$, which is sparse by \Cref{lem:57121}.
\end{proof}

\begin{theorem}\label{thm:29075}
    Let $f$ be a $p$-adic Hahn series satisfying the following conditions:
    \begin{enumerate}
        \item $f$ is algebraic over $\breve{\bfQ}_p$;
        \item $\opn{Supp}(f)$ is bounded and admits finitely many accumulation points.
    \end{enumerate}
    Then $\opn{Supp}(f)$ is a finite set.
\end{theorem}
\begin{proof}
    Suppose that $\opn{Supp}(f)$ is an infinite set. Then there exists at least one accumulation point of $\opn{Supp}(f)$.

    By \Cref{lem:42556}, we can write $\opn{Supp}(f)=W\sqcup S'$, where $W$ is a finite set, and $S'$ is an infinite set such that $-T(S'-\lambda)$ admits a nonzero sparse set of representatives modulo $\bfZ$ for some $\lambda\in\bfQ$ and integer $T\geq 1$. We write $f=f_0+f_1$ according to the decomposition $\opn{Supp}(f)=W\sqcup S'$, i.e., $f_0$ (resp. $f_1$) is the $p$-adic Hahn series consisting of the terms of $f$ with support in $W$ (resp. $S'$). Since $W$ is a finite set, $f_0$ is algebraic over $\breve{\bfQ}_p$. This forces $f_1$, and consequently $f_1\cdot p^{-\lambda}$, to be algebraic over $\breve{\bfQ}_p$. Note that the support of $f_1\cdot p^{-\lambda}$ is just $S'-\lambda$, and \Cref{thm:23791} implies that $f_1\cdot p^{-\lambda}$ is transcendental over $\breve{\bfQ}_p$, a contradiction.
\end{proof}

\begin{corollary}\label{coro:11594}
    Let $f\in\bfL_p$ be a $p$-adic algebraic number. Then $\opn{Supp}(f)$ has either no accumulation points or infinitely many accumulation points. In other words, its order type is either finite or no less than $\omega^2$.
\end{corollary}

\appendix
\section{AI-assisted formalization in Lean 4}\label{sec:formal}
\subsection{Overview of the formalization project}
The formalization of this paper, which we refer to as the \texttt{FormalizedSparse} project, contains approximately 24,000 lines of Lean code (including docstrings) and is organized into the following files:
\subsubsection{\texttt{WittVector.lean}}
This file contains our realization of $\breve{\bfQ}_p$, the completed maximal unramified extension of $\bfQ_p$, and its ring of integers $\breve{\bfZ}_p$.

Although Mathlib already contains a formalization of $\bfQ_p$, the lack of infrastructure for ramification theory in Mathlib makes it difficult to define $\breve{\bfQ}_p$ literally as the union of all finite unramified extensions of $\bfQ_p$. Instead, we start with $\breve{\bfZ}_p$, which we define as the ring of Witt vectors over $\overline{\bfF}_p$:
\begin{leancode}
-- The algebraic closure of `𝔽ₚ`
abbrev Fpbar (p : ℕ) [Fact (Nat.Prime p)] := AlgebraicClosure (ZMod p)
notation "𝔽ᵃ_[" p "]" => Fpbar p

-- The ring of integers of completed maximal unramified extension of `ℚₚ`,
abbrev OQpUn (p : ℕ) [Fact (Nat.Prime p)] := WittVector p (Fpbar p)
notation "ℤᵘⁿ_[" p "]" => OQpUn p
\end{leancode}
Then we define $\breve{\bfQ}_p$ as the fraction field of $\breve{\bfZ}_p$, with the induced valuation:
\begin{leancode}
abbrev QpUn (p : ℕ) [Fact (Nat.Prime p)] :=
  WithVal ((IsDiscreteValuationRing.maximalIdeal (ℤᵘⁿ_[p])).valuation 
    ((FractionRing (ℤᵘⁿ_[p]))))
notation "ℚᵘⁿ_[" p "]" => QpUn p
\end{leancode}

We establish several \lean{instance} around \lean{ℚᵘⁿ_[p]}, such as the fact that it is a complete rank-$1$ valued field. In addition, we define the embedding from $\bfQ_p$ to $\breve{\bfQ}_p$: one has $\bfZ_p\cong W(\bfF_p)$, which injects into $\breve{\bfZ}_p\coloneqq W(\overline{\bfF}_p)$. This extends to an embedding $\bfQ_p\hookrightarrow \breve{\bfQ}_p$:
\begin{leancode}
-- The embedding from ℚ_[p] to ℚᵘⁿ_[p].
noncomputable def Qp_embd {p : ℕ} [Fact (Nat.Prime p)] : ℚ_[p] →+* ℚᵘⁿ_[p] :=
  @IsFractionRing.map ℤ_[p] ℤᵘⁿ_[p] ℚ_[p] ℚᵘⁿ_[p] _ _ _ _ _ _ _ _ _
    ((WittVector.map (algebraMap (ZMod p) (𝔽ᵃ_[p]))).comp
      (WittVector.fromPadicInt p)) ... -- The map $\bfZ_p\lto\breve{\bfZ}_p$ is injective. Omitted.
\end{leancode}
Finally, we formalize a lemma to show that this embedding preserves the valuation.

\subsubsection{\texttt{PAdicHahnSeries.lean}}
This file contains the formalization of $\bfL_p$, the field of $p$-adic Hahn series. The material is mostly taken from \cite{poonenMAXIMALLYCOMPLETEFIELDS1993}.

Since the equal-characteristic Hahn series is already available in Mathlib, we formalize the ring $\breve{\bfZ}_p\pparen{t^\bfQ}$ as:
\begin{leancode}
abbrev LiftedPAdicHahnSeries (p : ℕ) [Fact (Nat.Prime p)] := HahnSeries ℚ (ℤᵘⁿ_[p])
\end{leancode}

We define the null series condition (cf. \Cref{prop:49318}) as a predicate \lean{IsNullSeries} on \lean{LiftedPAdicHahnSeries}, and define the ideal $\calN$ as the set of null series:
\begin{leancode}
def NullSeriesIdeal (p : ℕ) [Fact (Nat.Prime p)] : Ideal (LiftedPAdicHahnSeries p) where
  carrier := {x | IsNullSeries x}
  ... -- Omitted.
\end{leancode}
We provide an \lean{instance} to show that $\calN$ is a maximal ideal of $\breve{\bfZ}_p\pparen{t^\bfQ}$, so that the quotient $\breve{\bfZ}_p\pparen{t^\bfQ}/\calN$ is a field (cf. (1) of \Cref{prop:49318}).

Instead of directly defining $\bfL_p$ as the quotient of $\breve{\bfZ}_p\pparen{t^\bfQ}$ by $\calN$, we first formalize (2) of \Cref{prop:49318}:
\begin{leancode}
-- Every element of $\breve{\bfZ}_p\pparen{t^\bfQ}/\calN$ has a unique lift of the form $\sum_{q\in\bfQ}[f(q)]t^q$, 
-- with the support a well-ordered set.
theorem exists_canonical_expansion {p : ℕ} [Fact (Nat.Prime p)] :
  ∀ A : (LiftedPAdicHahnSeries p) ⧸ (NullSeriesIdeal p),
    ∃! (s : {f : ℚ → 𝔽ᵃ_[p] // (Function.support f).IsPWO}),
      Ideal.Quotient.ringCon (NullSeriesIdeal p)
      A.out (LiftedPAdicHahnSeries.from_coeff s.val s.prop)
\end{leancode}
This allows us to define a valuation on $\breve{\bfZ}_p\pparen{t^\bfQ}/\calN$ by considering the minimum of the support of this unique lift:
\begin{leancode}
def val (p : ℕ) [Fact (Nat.Prime p)] :
  AddValuation ((LiftedPAdicHahnSeries p) ⧸ (NullSeriesIdeal p)) (WithTop ℚ) where
  toFun x :=
    if h : x = 0 then (⊤ : WithTop ℚ) --If $x=0$, then its valuation is $\infty$.
    else ((support_IsPWO x).isWF.min (support_nonempty_of_nonzero p x h) : WithTop ℚ)
  ... -- Omitted.
\end{leancode}
And finally we define $\bfL_p$ as the fraction field of $\breve{\bfZ}_p\pparen{t^\bfQ}/\calN$:
\begin{leancode}
abbrev pAdicHahnSeries (p : ℕ) [Fact (Nat.Prime p)] : Type _ :=
  (LiftedPAdicHahnSeries p) ⧸ (NullSeriesIdeal p)
notation "𝕃_[" p "]" => pAdicHahnSeries p
\end{leancode}

Several facilities around $\bfL_p$, such as its support and the coefficients (as a function of type \lean{ℚ → 𝔽ᵃ_[p]}), are also formalized in this file.

To deliver the $\bfQ_p$-transcendence and $\breve{\bfQ}_p$-transcendence results, we provide the embedding from $\breve{\bfQ}_p$ to $\bfL_p$. It is induced by the embedding from $\breve{\bfZ}_p$ to $\bfL_p$, which maps an element $w$ to the image of $w\cdot t^0$ in the quotient $\breve{\bfZ}_p\pparen{t^\bfQ}/\calN$.
\begin{leancode}
def ZpUn_embd {p : ℕ} [Fact (Nat.Prime p)] : ℤᵘⁿ_[p] →+* 𝕃_[p] where
  toFun a := Ideal.Quotient.mk (NullSeriesIdeal p) (HahnSeries.single 0 a)
  ... -- Omitted.

def QpUn_embd {p : ℕ} [Fact (Nat.Prime p)] : ℚᵘⁿ_[p] →+* 𝕃_[p] :=
  IsFractionRing.map (j := ZpUn_embd (p := p)) ZpUn_embd_injective
\end{leancode}

Finally, we show by induction on the cardinality of the support that an element of $\bfL_p$ with finite support must be algebraic over $\bfQ_p$:
\begin{leancode}
lemma alg_of_fin_supp (p : ℕ) [Fact (Nat.Prime p)] (f : 𝕃_[p]) 
  (hf : f.support.Finite) : IsAlgebraic ℚ_[p] f
\end{leancode}

\subsubsection{\texttt{Tscaled.lean}}
This file corresponds to \Cref{sec:tscaled} of this paper. We use \lean{ℤᵘⁿ_[p,T]} (resp. \lean{ℚᵘⁿ_[p,T]}) to denote the ring $\breve{\bfZ}_{p,T}$ (resp. $\breve{\bfQ}_{p,T}$), and use \lean{𝕃_[p,T]} to denote the quotient of $\breve{\bfZ}_{p,T}\pparen{t^\bfQ}$ by $\calN_T$. After we formalize \Cref{lem:14623} and \Cref{lem:21624}, the isomorphism $\sigma$ in \Cref{prop:18592} can be delivered:
\begin{leancode}
def σ : 𝕃_[p] ≃+* 𝕃_[p,T] := ... -- Omitted.

-- $\text{\normalfont\Cref{prop:18592}}$, $\sigma$ is given by $\sum_{q\in\bfQ}[f(q)]p^q\mapsto \sum_{q\in\bfQ}[f(q)]p^q$.
theorem σ_coeff_compat (f : 𝕃_[p]) : (σ p T f).coeff = f.coeff
\end{leancode}

\subsubsection{\texttt{Sparse.lean}}
This file corresponds to \Cref{def:25951}, \Cref{def:15658} and \Cref{sec:sparse} of this paper, which is about the sparseness, $(c,n)$-sparseness and the related infrastructure.

Fortunately, Mathlib already contains the function \lean{Real.digits}, which extracts the digits of a real number in a given base. This allows us to formalize the $p$-digit sum of a rational number, and consequently $\opn{dom}_p(S)$ and $\opn{Dom}_p(S)$ of a set $S\subseteq\bfQ$ with relatively little effort:
\begin{leancode}
abbrev decDigits (p : ℕ) [Fact (Nat.Prime p)] (q : ℚ) : ℕ+ → Fin p :=
  fun n => Real.digits (Int.fract q) p ((n : ℕ) - 1)

/- The `p-digit sum`, `𝔑ₚ(q)` in $\text{\normalfont\Cref{def:25951}}$ (1)-/
def pDigitSum (p : ℕ) [Fact (Nat.Prime p)] (q : ℚ) : WithTop ℕ :=
  if h : (Function.support (decDigits p q)).Infinite then ⊤
  else ∑ n ∈ (Set.not_infinite.1 h).toFinset, (decDigits p q n).val

/- `dominant p-digit sum` of S, $\text{\normalfont\Cref{def:25951}}$ (2)-/
def dom (p : ℕ) [Fact (Nat.Prime p)] (S : Set ℚ) : WithTop ℕ :=
  sSup {pDigitSum p  q | q ∈ S}

/- `p-digit dominant part` of S, $\text{\normalfont\Cref{def:25951}}$ (2)-/
def Dom (p : ℕ) [Fact (Nat.Prime p)] (S : Set ℚ) : Set ℚ :=
  {q ∈ S | pDigitSum p q = dom p S}
\end{leancode}

With these preparations, we present the formalized version of \Cref{def:15658}:
\begin{leancode}
/- The sparse condition, as a predicate -/
def IsSparse (p : ℕ) [Fact (Nat.Prime p)] (S : Set ℚ) : Prop :=
  S ⊆ Set.Ico 0 1 ∧ dom p S < ⊤ ∧
  ∃ D : Set ℕ+, D.Infinite ∧ ( ∀ n ∈ D, ∃ d : Fin n → Dom p S, (
      ∀ i : ℕ+, ∑ (j : Fin n), (decDigits p (d j) i).val < p
      -- No carrying when adding d₁, d₂, ..., dₙ together.
    ) ∧ ( ∀ e : Fin n → Dom p S,
      (∑ i, (d i).val -∑ i, (e i).val).isInt →
        ∃ perm : Equiv.Perm (Fin n), ∀ i, d i = e (perm i)
      -- Unique up to a permutation of $\{1,\cdots n\}$.
    ) )
\end{leancode}

Since \Cref{sec:sparse} contains mostly implementation-level details, we will not demonstrate most of the formalization here, except for \Cref{eg:new}: the $p$-digit disjoint subset of $[0,1)\cap\bfQ$ is sparse:
\begin{leancode}
lemma IsSparse_of_digit_disjoint (p : ℕ) [Fact (Nat.Prime p)] (A : ℕ → Set ℕ+)
(hA1 : ∀ n, (A n).Nonempty) (hA2 : ∀ i j, (A i) ∩ (A j) ≠ ∅ → i = j)
(hA3 : ∀ n, (A n).Finite)
(hAsup : ∃ K : ℕ, (∀ n, (hA3 n).toFinset.card ≤ K) ∧
          {n | (hA3 n).toFinset.card = K}.Infinite) :
  IsSparse p {∑ r ∈ (hA3 i).toFinset, (p : ℚ) ^ (-(r: ℤ)) | i : ℕ }
\end{leancode}

\subsubsection{\texttt{MainTheorem.lean}}
The single objective of this file is to formalize the proof of \Cref{thm:main}:
\begin{leancode}
theorem main_theorem (p : ℕ) [Fact (Nat.Prime p)] (f : 𝕃_[p]) (T : ℕ+)
(W : Set ℚ) (hW1 : W ≠ {0}) (hW2 : IsSparse p W)
(hf : IsRepModZ W {-1 * T * q | q ∈ f.support}) :
  ¬ IsAlgebraic ℚᵘⁿ_[p] f
\end{leancode}
Here \lean{IsRepModZ} is a predicate to express the condition that $W$ is a set of representatives of $-T\cdot\opn{Supp}(f)$ modulo $\bfZ$.

A variant of this theorem, which replaces the algebraicity over $\breve{\bfQ}_p$ by the algebraicity over $\bfQ_p$, is also formalized in this file.

\subsubsection{\texttt{QuasiTwistRecurrent.lean}}
This file corresponds to \Cref{sec:5330}. The key concept is that of a quasi-twist-recurrent (QTR) function (\Cref{def:16557}). Throughout, the base-$p$ digit string $0.q_1\cdots q_n\cdots=\sum_{i\geq 1}q_ip^{-i}$ is modeled by a finitely supported \lean{d : ℕ →₀ ℕ}, where \lean{d i} is the digit $q_{i+1}$. It is implemented as a predicate as follows:
\begin{leancode}
def IsQTR {p : ℕ} [Fact (Nat.Prime p)] (x : ℚ → 𝔽ᵃ_[p])
(a : ℕ+) (b c : ℕ) (M N : ℕ+) : Prop :=
  (Function.support x).IsWF ∧ -- The support is a well-ordered subset of $\mathbf{Q}$.
  (Function.support x ⊆ Kedlaya.Sabc p a b c) ∧ -- The support is contained in $S_{a,b,c}$.
  -- The $M$-zero-run ⇒ $N$-zero-insertion recurrence.
  ∀ (w : ℤ), -(b : ℤ) ≤ w → ∀ (d : ℕ →₀ ℕ), (∀ i, d i < p) →
      (d.sum fun _ v => v) ≤ c →
    ∀ (k : ℕ), (∀ i, k ≤ i → i < k + (M : ℕ) → d i = 0) →
      x ((1 / (a : ℚ)) * ((w : ℚ) - d.sum fun i v => 
        (v : ℚ) * (p : ℚ) ^ (-(i + 1 : ℤ)))) = x ((1 / (a : ℚ)) * ((w : ℚ) -
          (Finsupp.mapDomain (fun i => 
            if i < k + (M : ℕ) then i 
            else i + (N : ℕ)) d).sum fun i v => (v : ℚ) * (p : ℚ) ^ (-(i + 1 : ℤ))))
\end{leancode}
Here \lean{Kedlaya.Sabc p a b c} is the set $S_{a,b,c}$ of \Cref{thm:28713}, and the insertion of $N$ zeros into a length-$M$ zero gap is realized by \lean{Finsupp.mapDomain}. Building on Kedlaya's criterion (\Cref{thm:28713,thm:47}), the main result of this file is the necessary QTR condition for bounded $\breve{\bfQ}_p$-algebraic $p$-adic Hahn series (\Cref{prop:167}):
\begin{leancode}
theorem isQTR_of_isAlgebraic_of_bddSupport {p : ℕ} [Fact (Nat.Prime p)]
(f : 𝕃_[p]) (halg : IsAlgebraic ℚᵘⁿ_[p] f) (hbdd : IsBounded f.support) :
  ∃ (a : ℕ+) (b c : ℕ) (M N : ℕ+), IsQTR f.coeff a b c M N
\end{leancode}

\subsubsection{\texttt{RayDecomposition.lean}}
This file corresponds to \Cref{sec:23337}. The central concept is that of a ray (\Cref{def:33761}). Using the same finsupp digit model as above, a ray with base element \lean{d_base} is obtained by inserting $kN$ zeros at a fixed gap position \lean{shift_pos} for $k=0,1,2,\cdots$. It is formalized as a predicate as follows:
\begin{leancode}
def IsRay (p : ℕ) [Fact (Nat.Prime p)] (a : ℕ+) (c : ℕ) (m : ℤ) (N : ℕ+) (S : Set ℚ) 
(R : Set ℚ) : Prop := R ⊆ S ∧
  ∃ (d_base : ℕ →₀ ℕ) (shift_pos : ℕ),
    (∀ i, d_base i < p) ∧ (d_base.sum fun _ v => v) ≤ c ∧
    R = { q : ℚ | ∃ k : ℕ,
      q = (1 / (a : ℚ)) * ((m : ℚ) -
        (Finsupp.mapDomain
          (fun i => if i < shift_pos then i else i + k * (N : ℕ)) d_base).sum
          fun i v => (v : ℚ) * (p : ℚ) ^ (-(i + 1 : ℤ))) }
\end{leancode}
The main result of this file is the ray decomposition of bounded QTR sets (\Cref{coro:48108}): a bounded QTR set with finitely many accumulation points is a union of a finite set and finitely many pairwise disjoint rays. Accumulation points are formalized via \lean{derivedSet}, and \lean{rays.sup id} denotes the union of the finite family \lean{rays}:
\begin{leancode}
/- Ray decomposition of bounded QTR sets -/
theorem qtr_ray_decomposition {p : ℕ} [Fact (Nat.Prime p)]
{a : ℕ+} {b c : ℕ} {M N : ℕ+} {x : ℚ → 𝔽ᵃ_[p]}
(hqtr : IsQTR x a b c M N) (hbdd : Bornology.IsBounded (Function.support x))
(hfin_acc : (derivedSet (Function.support x)).Finite) :
  ∃ (E : Finset ℚ) (rays : Finset (Set ℚ)),
    (∀ R ∈ rays, ∃ (m : ℤ) (hm : -(b : ℤ) ≤ m),
      IsRay p a c m N (Function.support x) R ∧
        IsAdmissible p a (b : ℤ) c m hm M N (Function.support x ∩ Sabc_m p a c m)) ∧
      (∀ R₁ ∈ rays, ∀ R₂ ∈ rays, R₁ ≠ R₂ → Disjoint R₁ R₂) ∧
      Function.support x = ↑E ∪ rays.sup id
\end{leancode}
Here \lean{IsAdmissible} and \lean{Sabc_m} are the admissibility conditions (S1)--(S3) and the $m$-slice $S_{a,b,c,m}$ from \Cref{sec:23337}, whose formalization we omit.

\subsubsection{\texttt{BoundedSupport.lean}}
This file corresponds to the subsection \Cref{sec:24235}. Combining the ray decomposition with the sparseness of the explicit witness set (the formalized version of \Cref{lem:57121}) and the main transcendence theorem \lean{main_theorem} (\Cref{thm:23791}), we formalize the finiteness of bounded QTR supports (\Cref{thm:29075}):
\begin{leancode}
theorem finite_support_of_qpun_algebraic_of_bounded_support 
{p : ℕ} [Fact (Nat.Prime p)]
(f : 𝕃_[p]) (hf1 : IsAlgebraic ℚᵘⁿ_[p] f) (hf2 : IsBounded f.support)
(hf3 : (derivedSet ((Rat.cast : ℚ → ℝ) '' f.support)).Finite) :
  f.support.Finite
\end{leancode}
And consequently \Cref{coro:11594}:
\begin{leancode}
-- $\bfQ_p$-algebraic + bounded support ⇒ $\opn{card}\{\text{accumulation points}\}=$ $0$ or $\infty$.
theorem support_accpt_empty_or_infinite_of_qp_algebraic_of_bounded_support
{p : ℕ} [Fact (Nat.Prime p)] 
(f : 𝕃_[p]) (hf1 : IsAlgebraic ℚ_[p] f) (hf2 : IsBounded f.support) :
  (derivedSet ((Rat.cast : ℚ → ℝ) '' f.support)) = ∅ ∨
  (derivedSet ((Rat.cast : ℚ → ℝ) '' f.support)).Infinite

-- The order type variant.
theorem order_type_of_qp_algebraic_of_bounded_support {p : ℕ} [Fact (Nat.Prime p)]
    (f : 𝕃_[p]) (hf1 : IsAlgebraic ℚ_[p] f) (hf2 : IsBounded f.support) :
    typeLT f.support < omega0 ∨ typeLT f.support ≥ omega0^2
\end{leancode}

We also formalize a literal version of \Cref{thm:16129} in this file, which we omit here for brevity.

\subsubsection{\texttt{HuangStefanescu.lean}}
This file contains the formalization of two quick corollaries of \Cref{thm:16129}: \Cref{coro:3344} and \Cref{prop:31772}.
\begin{leancode}
-- $\bfQ_p$-algebraic + strictly monotone support ⇒ unbounded support
theorem tendsto_atTop_of_strictMono_support_of_qpun_algebraic
{p : ℕ} [Fact (Nat.Prime p)] (s : ℕ → ℚ) (hs : StrictMono s) (f : 𝕃_[p])
(hf1 : IsAlgebraic ℚᵘⁿ_[p] f) (hf2 : f.support = Set.range s) :
  Tendsto s atTop atTop

-- The $p$-adic analogue of the result of Huang and Stefanescu.
theorem padic_huang_stefanescu_tfae (p : ℕ) [Fact (Nat.Prime p)] 
(f : 𝕃_[p]) (hf : f.support ⊆ {-(p : ℚ) ^ (-(i : ℤ)) | i : ℕ+}) :
  List.TFAE [f.support.Finite, IsAlgebraic ℚᵘⁿ_[p] f, IsAlgebraic ℚ_[p] f]
\end{leancode}

\subsubsection{\texttt{Kedlaya.lean}}
This file records the deep results of Kedlaya (\Cref{thm:28713,thm:47}) that we use as black boxes. Reproving them is outside the scope of this formalization, so their statements are given here and their proofs are the only intentional \lean{admit}s in the project. The shared dependency is the support set $S_{a,b,c}$ of \Cref{thm:28713}, where the base-$p$ digit sequence is again modeled by a finsupp \lean{d : ℕ →₀ ℕ}:
\begin{leancode}
/- $\text{\normalfont\cite[Definition 2.1]{kedlayaAlgebraicityGeneralizedPower2017b}}$ -/
def Sabc (a : ℕ+) (b c : ℕ) : Set ℚ :=
  { s : ℚ | ∃ (n : ℤ) (d : ℕ →₀ ℕ),
    -b ≤ n ∧ (∀ i, d i < p) ∧ (d.sum fun _ v => v) ≤ c ∧
    s = (1 / (a : ℚ)) *
      ((n : ℚ) - d.sum fun i v => (v : ℚ) * (p : ℚ) ^ (-(i + 1 : ℤ))) }
\end{leancode}
The integrality criterion (\Cref{thm:28713}) is then stated as follows, where \lean{(𝔽ᵃ_[p])⸨X⸩} is the Laurent series field $\overline{\bfF}_p\pparen{t}$; note that integrality and algebraicity coincide over a field, so we phrase downstream results in terms of algebraicity:
\begin{leancode}
/- $\text{\normalfont\cite[Theorem 11.11]{kedlayaAlgebraicityGeneralizedPower2017b}, \cite[Theorem 15]{kedlayaAlgebraicClosurePower2001}}$ -/
theorem kedlaya_2001a_theorem15 (x : HahnSeries ℚ (𝔽ᵃ_[p])) :
  IsIntegral (𝔽ᵃ_[p])⸨X⸩ x ↔ ∃ a : ℕ+, ∃ b c : ℕ,
    (x.support ⊆ Sabc p a b c) ∧ -- $x$ is supported on some $S_{a,b,c}$.   
    (∃ M N : ℕ+, ...) -- the twist functions $f_m$ are eventually periodic, omitted.
    := by admit
\end{leancode}
The second black box is Kedlaya's description of the completed integral closure of $\breve{\bfQ}_p$ in $\bfL_p$ (\Cref{thm:47}), which is the input to \Cref{prop:167}:
\begin{leancode}
/- $\text{\normalfont\cite[Theorem 13.4]{kedlayaAlgebraicityGeneralizedPower2017b}}$ -/
theorem kedlaya_2017_theorem13_4 :
  closure (integralClosure ℚᵘⁿ_[p] 𝕃_[p]).carrier =
  closure { f : 𝕃_[p] | ∃ f' : HahnSeries ℚ (𝔽ᵃ_[p]), IsAlgebraic 𝔽ᵃ_[p] f' ∧
    (exists_canonical_expansion f).choose.val = f'.coeff } := by admit
\end{leancode}

\subsection{AI-assisted formalization}
The agentic auto-formalization system Archon, which is based on Claude Opus 4.8 and developed by the AI4Math team at BICMR, Peking University, greatly accelerated our formalization process. By design, once the blueprint or the corresponding informal proof of the project is provided, Archon works fully autonomously to complete the project-level informal-to-formal translation of the statements (as well as mathematical definitions) and to formalize the proof.

\subsubsection{Good translation of the statements and definitions}
Specific to our project, our experience suggests that it is better to write all the formalized definitions and statements by a human mathematician with solid experience in Lean\footnote{In the \texttt{FormalizedSparse} project, we formalized all the definitions and statements by hand.}, and then let Archon fill in the proofs. While Archon is capable of independently formalizing mathematical concepts, it, like other artificial intelligence systems, lacks adequate mathematical intuition to formulate definitions in a manner that facilitates their subsequent application.

For example, at the beginning of the project, we let Archon formalize $\breve{\bfQ}_p$, the completed maximal unramified extension of $\bfQ_p$. As we observed, Archon made considerable efforts to formalize the definition of $\breve{\bfQ}_p$ literally, but without success. In fact, the mathematical insight here is that the ramification information is not what is actually needed. Instead, the key property of $\breve{\bfQ}_p$ that is repeatedly used in this project is that every element of $\breve{\bfQ}_p$ can be uniquely written as a $p$-adic Laurent series whose coefficients are Teichmüller representatives. This is the perfect scenario for applying the ring of Witt vectors, which is already formalized in Mathlib.

\subsubsection{Automatic formalization: proof and revision}
After the formalized statements and definitions are settled\footnote{We also built up several helper results (with \lean{sorry}) that we expected to be useful during the formalization process.}, we let Archon fill in the proofs without supervision. The results are quite satisfactory: Archon is able to complete most proofs in a way that closely follows the informal proof, and the generated code (as well as the docstrings) is mostly readable. We learned several things from this process:
\begin{enumerate}
    \item During its work, the formalization by Archon helped us find several mistakes (which are now fixed) and subtleties in the original version of the informal proof:
    \begin{enumerate}
        \item \Cref{thm:main} does not hold for the sparse set $W=\{0\}$.
        \item We did not require the supremum in \Cref{eg:61947} to be attained by infinitely many $i$ in the original version, which is also necessary for the proof to work.
        \item As we mentioned at the beginning of \Cref{sec:tscaled}, some elements in $\bfL_p$ were casually written in the form $\sum_{q\in [0,1/T)}c_q p^q$, with $c_q\in\breve{\bfQ}_{p,T}$ for every $q$. This is not a well-defined element of $\bfL_p$, unless we consider the $T$-scaled realization (i.e., \lean{𝕃_[p,T]}) of $\bfL_p$ and the isomorphism $\sigma$ in \Cref{prop:18592}.
        \item The phrase ``accumulation points'' could be ambiguous for a set of rational numbers (cf. \Cref{rmk:37294}).
    \end{enumerate}
    These fragility, all of which have been fixed in the current version, are related to the technical details of the proof, and would not have been easy to find without the help of formalization. Archon recorded these mistakes in the provisional docstrings. For the first two mistakes, Archon added the necessary condition by itself and continued the formalization process without any human intervention. For the other two issues, Archon failed to complete the formalization and terminated with a detailed report.
    \item Archon's ability to backtrack is impressive. When formally proving \Cref{thm:main}, which is highly combinatorial, we gave no hint to Archon about the structure of the proof, except for the informal proof itself. Archon made multi-level plans to divide the proof into several lemmas and assemble them to complete the proof. During the work of Archon, we observed that some of its intermediate lemmas were incorrect, and Archon was able to backtrack and revise the proof plan to fix the mistake without any human intervention.
\end{enumerate}

\subsection{How formalization helps mathematical research}
For mathematicians with limited experience in formalization, agent-based systems such as Archon may eventually provide a practical way to validate proofs in a largely black-box manner. One can envision a future workflow in which a paper written in natural language is automatically translated into a formal proof object and subsequently verified by the system after extensive computation.

For mathematicians with some experience in formalization, we believe that mathematical research can benefit substantially from a human-in-the-loop workflow. In such a workflow, researchers formulate definitions and statements of intermediate results in formal language, while AI systems assist with the labor-intensive formalization process. Successfully formalized intermediate results then provide verified foundations for subsequent arguments, while failed formalizations may help detect errors at an early stage of the research process.

\printbibliography
\end{document}